\documentclass[11pt]{article}
\usepackage{macros}

\usepackage[most]{tcolorbox}
\tcbset{
	emphbox/.style={
		enhanced,
		colback=white,
		colframe=black,
		boxrule=0pt,
		sharp corners,
		frame hidden,
		left=2.4mm,
		top=0mm,
		bottom=0mm,
		borderline west={0.8mm}{0pt}{black},
		before skip=3pt,
		after skip=3pt,
	}
}

\bibliography{bibliography.bib}

\title{Toposes with enough points as categories of étale spaces}
\author{Sam van~Gool, Jérémie Marquès and Umberto Tarantino}
\date{\today}

\begin{document}
\maketitle

\begin{abstract}
    We extend Makkai duality between coherent toposes and ultracategories to a duality between toposes with enough points and ultraconvergence spaces. Our proof generalizes and simplifies Makkai's original proof. Our main result can also be seen as an extension to ionads of Barr's equivalence between topological spaces and relational modules for the ultrafilter monad. In view of the correspondence between toposes and geometric theories, we obtain a strong conceptual completeness theorem, in the sense of Makkai, for geometric theories with enough Set-models. The same result has recently been obtained independently by Saadia (\href{https://arxiv.org/abs/2506.23935}{arXiv:2506.23935}) and by Hamad (\href{https://arxiv.org/abs/2507.07922}{arXiv:2507.07922}). Both of their proofs rely on groupoid representations of toposes, which our proof here does not assume.  
\end{abstract}

\tableofcontents

\section{Introduction}\label{introduction}
A fundamental idea in Stone duality theory~\cite{Stone1937} is that a topological space may be
studied via its open subsets, which have the algebraic structure of a \emph{frame}:
a partially ordered set which admits finite infima and arbitrary suprema satisfying a distributive law.
Similarly, sheaves of sets on a topological space have the
algebraic-categorical structure of a (Grothendieck) \emph{topos}: a category admitting
finite limits and arbitrary colimits satisfying some exactness properties. This paper is a
contribution to the study of extensions of Stone duality from the
\emph{propositional} setting of frames to the \emph{first-order} setting of
toposes.

An abstract justification for the study of topology via frames lies in the fact
that the contravariant functor $\mc O \colon \cat{Top}^{\op} \to \cat{Frm}$
mapping a topological space to its frame of opens is part of an adjunction~\cite{PapertPapert1957,Isbell1972}: its
left adjoint sends a frame $F$ to the set $X$ of frame homomorphisms from $F$
to the two-element frame $\two$, suggestively called \emph{points} of $F$. Any
element $\phi$ of $F$ gives an \emph{evaluation} function $\sem{\phi} \colon X
\to \two$, sending a point $x$ to $x(\phi)$, and these maps determine a frame
homomorphism $\sem{-} \colon F\to \two^X$ whose image $\Oc (X)$ defines a topology
on $X$. The surjection $F \onto \mc O(X)$ is then the unit of the adjunction:
the frame $F$ is called \emph{spatial} if this unit is injective, and thus an
isomorphism.

A similar story can be told for toposes, where the topos $\Set$ plays the role
of the frame $\two$. A \emph{point} of a topos $\mc E$ is a functor $\mc E \to \Set$ 
which preserves finite limits and small colimits. 
For a class $X$ of points of $\mc E$, any object $\phi$ of $\mc E$ gives an 
{evaluation} map $\sem\phi \colon X \to \Set$, sending a point $x$ to
$x(\phi)$, and these maps define a functor $\sem{-} \colon \mc E \to \Set^X$.
In analogy with topological spaces, a \emph{bounded ionad} \cite{garner-ionads} is a pair $(X, \Ec)$ where $X$ is a separating class of points of a topos $\Ec$, which means that the functor $\sem{-} \colon \Ec \to \Set^X$ is conservative, i.e., reflects isomorphisms. %
When a class of points $X$ such that $\sem{-}$ is conservative exists, the topos $\Ec$ is said to \emph{have enough points}. This is the topos-theoretic analogue of spatiality for frames. From a logical point of view, toposes with enough points correspond to geometric theories enjoying a completeness theorem with respect to their class of $\cSet$-models.

In this paper, we will equip a separating class of points of a topos with enough additional structure so that the topos can be reconstructed from it. In the propositional setting, Barr~\cite{barr} showed that a topology on a set $X$ can be reconstructed from a convergence relation between $X$ and the ultrafilters on $X$. 
In the first-order setting, Makkai~\cite{makkai} showed that \emph{ultraproducts} may be used to obtain such a reconstruction result for pretoposes, and this was later extended to coherent toposes by Lurie \cite{lurie}.
For an $I$-indexed family of sets $(S_i)_{i \in I}$
and an ultrafilter $\mu$ on $I$, we write $\prod_{i \to \mu} S_i$ for the
ultraproduct of the family with respect to $\mu$; we call its elements
\emph{ultrafamilies indexed by $\mu$}, or \emph{$\mu$-families}, and we denote them as $(s_i)_{i \to \mu}$. 
The ultraproduct construction yields a functor $\prod_{i\to \mu} \colon \Set^I
\to \Set$: the model-theoretical {\L}o{\'s}'s theorem, transported to this
context, says that the ultraproduct functor is \emph{coherent}: that is, it preserves finite limits, regular epimorphisms, and finite unions of subobjects. Thus, if $(x_i)_{i \in I}$ is an $I$-indexed family of coherent functors
from a coherent category $\mc C$ to $\Set$, then the composite functor 
\[\prod_{i \to \mu} x_i \colon \begin{tikzcd}[sep = large]
	{\mc C} & {\Set^I} & \Set
	\arrow["{\braket{x_i}_{i\in I}}", from=1-1, to=1-2]
	\arrow["{\prod_{i\to\mu}}", from=1-2, to=1-3]
\end{tikzcd}\]
is coherent. 

\looseness=-1
To make the link with toposes, recall that we may think of the points of a topos $\mc E$ as the models of a logical theory classified by $\mc E$. If this theory is \emph{coherent}, then {\L}o{\'s}'s theorem gives that the ultraproduct of any family of points of $\mc E$ is again a point of $\mc E$. Therefore, the class of points of a coherent topos can be equipped with the structure of `taking ultraproducts', making it an \emph{ultracategory}.
An underlying idea here
is that ultraproducts are analogous to limits in topology: as shown in~\cite{lurie}, ultracategories with a
small set of objects and a discrete underlying category coincide with compact Hausdorff spaces, or, in light of Manes'
theorem~\cite{manes-comphaus}, algebras for the ultrafilter monad $\beta$ on $\Set$. The main results of~\cite{makkai, lurie} in this context are reconstruction results for pretoposes and coherent toposes, representing them as categories of functors to $\Set$ that preserve ultraproducts in an appropriate sense.

To generalize from coherent toposes to arbitrary toposes or, in logical terms, from coherent to \emph{geometric} theories, the main complication is that the ultraproduct of a
set of points is not necessarily a point itself. We can see this already at the propositional level, simply because an ultrafamily of points in a topological space may fail to have a unique limit.
To overcome this, given a class $X$ of points of a topos $\Ec$, we will consider the structure defined by formal ``maps into an ultraproduct'', even when the ultraproduct is not a point of $\Ec$. More precisely, for $x \in
X$ and $(x_i)_{i \to \mu}$ a $\mu$-family of points in $X$, we will consider natural transformations from the point
$x$ to the functor $\prod_{i \to \mu} x_i \colon \Ec \to \Set$, and denote the set
of these by $\Hom_{\ult}(x, (x_i)_{i \to \mu})$. We call this the
\emph{canonical ultraconvergence structure} on the class $X$, and we will refer to elements of
$\Hom_{\ult}(x, (x_i)_{i \to \mu})$ as \emph{ultra-arrows}.

In the propositional case, if $X$ is the set of points of a frame $F$, this
canonical ultraconvergence structure is just an encoding of the topology on $X$ described above.  Indeed, the assignment $\Hom_{\ult}$ then takes values in $\two$, and
may thus be seen as a relation between points in $X$ and ultrafamilies in $X$, where a point $x$ is related to a $\mu$-family of points $(x_i)_{i
\to \mu}$ if, and only if, every open neighborhood of $x$ also contains a $\mu$-large subset of the family $(x_i)_{i
\to \mu}$.

Coming back to the setting of a general topos $\mc E$ and a class $X$ of points
of $\mc E$ endowed with the canonical ultraconvergence structure, the evaluation map
$\sem\phi \colon X \to \Set$ associated to an object $\phi$
of $\mc E$ now has a natural action on ultra-arrows: given an
ultra-arrow $r$ from a point $x$ to a $\mu$-family $(x_i)_{i \to \mu}$, 
its component at $\phi$ gives
a function $r_{\phi} \colon x(\phi) \to \prod_{i \to \mu} x_i(\phi)$, that is,
an ultra-arrow for the ultraconvergence structure on the class $\Set$.
We call a map $X \to \Set$ equipped with such an action on ultra-arrows \emph{continuous}. 
This assignment defines a functor $\eval{-}$ from $\Ec$ to the category $\mathbf{C}(X,\Set)$ of continuous maps $X \to \Set$, 
which plays the same role as the surjection $F \onto \Oc(X)$ in the propositional case.
The main result of this paper, \Cref{thm:main-theorem} below, is that $\eval{-}$ is always an equivalence for toposes with enough points.

\begin{restatable}{theorem}{mainresult}\label{thm:main-theorem}
    Let $\mc{E}$ be a topos and let $X$ be a separating set of points of $\mc{E}$, equipped with the canonical ultraconvergence structure. Then $\eval{-}$ is an equivalence between $\mc{E}$ and the category of continuous maps from $X$ to $\Set$. 
\end{restatable}

\looseness=-1
Theorem~\ref{thm:main-theorem} builds on and generalizes work of \cite{makkai, lurie}, and it was also recently obtained independently in~\cite{saadia2025} and in~\cite{hamad2025}; we will review related work in more detail below. The aim of this paper is to give an elementary and self-contained proof of the
theorem. 
The key novelty of our approach lies in its use of \emph{étale maps}
of ultraconvergence spaces (Definition~\ref{def:etale-map}), extending the notion of étale map from the
topological context.
As we will establish in Theorem~\ref{thm:etale-spaces-as-continuous-maps}, 
by extending the usual correspondence
between discrete opfibrations and copresheaves, étale maps over an ultraconvergence
space $X$ are equivalent to the continuous maps from $X$ to $\Set$.
Post-composing the evaluation functor $\sem{-}$ with this equivalence, every
object $\phi$ of $\Ec$ gives rise to an étale space $\sem{\phi} \to X$, whose
fiber at a point $x \in X$ is the set $x(\phi)$:
\[ \begin{tikzcd}
	& \Ec \\
	{\mathbf{C}(X,\Set)} && {\Et(X).}
	\arrow["{\eval{-}}"', from=1-2, to=2-1]
	\arrow["{\eval{-}}", from=1-2, to=2-3]
	\arrow["\simeq", from=2-1, to=2-3]
\end{tikzcd} \]
We will prove
that the functor $\eval{-}\colon \Ec\to\Et(X)$ is an
equivalence, which then also gives the main theorem in light of Theorem
\ref{thm:etale-spaces-as-continuous-maps}.

\paragraph{Related work.}
An old idea in topology is that a space can be encoded as a convergence relation between points and generalized sequences \cite[Ch.~2, Thm.~9]{Kelley1955}. Building on that idea, Barr~\cite{barr} showed that a topology on a set $S$ can equivalently be described as the structure of a
\emph{relational $\beta$-module} on $S$, that is, a suitable relation between
points of $S$ and ultrafilters on $S$. We will see in Remark \ref{rem:top-sp-from-ultraconv-structure}
how our ultraconvergence relation, in the topological case, restricts to Barr's
description.

The main idea of our \emph{ultraconvergence spaces} is to replace Barr's
two-valued relation, between points of a space and ultrafilters on the space,
with a $\Set$-valued `relation', between points of a space and
\emph{ultrafamilies} in the space. The same intuition is present in Saadia's
\emph{virtual ultracategories}~\cite{saadia2025} and Hamad's \emph{generalized
ultracategories}~\cite{hamad2025}, two recent, independent works on the same
topic. 
Much earlier, Burroni \cite{burroni} already attempted to categorify Barr's
ultraconvergence relation, which led to a definition of \emph{$T$-categories}
for a monad $T$. For $T$ the \emph{ultrafilter monad} $\beta$ on $\Set$, he
dubbed $T$-categories \emph{hypertopologies}, and recovered Barr's relational
$\beta$-modules as the `posetal' case. Other, related approaches are
\cite{clementino03, clementino04}, which were later put in the framework of Cruttwell and Shulman's
\emph{generalized multicategories} \cite{cruttwell2010}. 
The connection with \cite{cruttwell2010}, and hence implicitly with Burroni's
$T$-categories, is also made in~\cite{saadia2025,hamad2025}.

A different but related direction in categorifying topological spaces is given
by Garner's \emph{ionads} \cite{garner-ionads}, later revisited by Di Liberti
\cite{diliberti22,diliberti24}, who showed an equivalence between the
2-category of toposes with enough points and that of sober (bounded) ionads. In
this paper, we typically consider a topos equipped with a choice of separating
set of points, which is an equivalent way of presenting a bounded ionad (cf.\
\cite[Rem.~2.5]{garner-ionads}). We leave the precise connection to ionads
implicit.

Our main theorem was proved around the same time in both \cite{saadia2025} and
\cite{hamad2025}, but with rather different techniques. The proofs in \cite{saadia2025,hamad2025} make essential use of Butz and Moerdijk's
 representation theorem for toposes with enough points as
categories of equivariant sheaves over topological groupoids constructed from
their points (\cite{butz-moerdijk}, see also \cite{joyal-tierney, awodey-forsell,
wrigley-topologicalgroupoids}). Our work here, on the other hand, does not rely on  
such a representation of the topos $\Ec$. 
A corollary of our main theorem is a different representation theorem for
toposes with enough points, independent of Butz and Moerdijk's: any topos with
enough points $\Ec$ is equivalent to the category $\Et(\pt(\Ec))$ of étale
spaces over $\pt(\Ec)$ endowed with the canonical ultraconvergence structure
(Corollary~\ref{cor:reconstruction}).

\paragraph{Outline.}\label{par:outline} We first recall some category-theoretic
preliminaries in Section~\ref{sec:cat-prelim}. In Section~\ref{sec:ultraconvergence-spaces}, 
we introduce the 2-category of \emph{ultraconvergence
spaces} (Definitions~\ref{dfn:uc-space}~and~\ref{def:continuous-maps}), 
which categorifies the description of a
topological space as a relational $\beta$-algebra.
In Section~\ref{sec:etale-spaces} 
we introduce the main conceptual novelty of this paper, namely a
notion of \emph{étale map} of ultraconvergence spaces (Definition~\ref{def:etale-map}). 
We prove in Section~\ref{subsec:etale-general} a number of properties of these maps, 
including the fact that they correspond to local homeomorphisms in the topological case
(Corollary \ref{cor:étale-is-local-homeo}). 
We then show in Section~\ref{subsec:grothendieck} that for each
ultraconvergence space $X$, the category of continuous maps
$X\to\Set$ is equivalent to the category $\Et(X)$ of étale spaces over $X$
(Theorem~\ref{thm:etale-spaces-as-continuous-maps}), extending the usual
correspondence between copresheaves and discrete opfibrations. We show in Section~\ref{subsec:etale-pretopos} that $\Et(X)$ is an infinitary pretopos.
Section~\ref{sec:main} contains the core of the paper. 
We first describe the canonical ultraconvergence structure
on a class of points $X$ of a topos $\Ec$,
and we then prove that the functor $\eval{-}
\colon \Ec \to \Et(X)$ is an equivalence of categories if $X$ is a separating set.
This proof is divided into two steps: $\eval{-}$ is full on subobjects (Proposition~\ref{prop:step-1}) and covering (Proposition~\ref{prop:step-2}).
As we will see with a simple argument, this is sufficient to ensure that $\eval{-} \colon \Ec \to \Et(X)$ is an equivalence of categories. Finally, in Section \ref{sec:conclusions}, we derive the representation theorem for toposes with enough points (Corollary~\ref{cor:reconstruction}) and make some concluding remarks.

\paragraph{Acknowledgements.} SvG and UT acknowledge financial support from the Agence Nationale de la Recherche (ANR), project ANR-23-CE48-0012-01. The authors thank Gabriel Saadia, Joshua Wrigley, and Errol Yuksel for useful discussions on the topic of this paper; see also Remark~\ref{rem:completing-cospans-ack}. The authors also thank the anonymous reviewers of an earlier draft of this paper for their valuable feedback.

\section{Category-theoretic preliminaries}\label{sec:cat-prelim}
Throughout the paper, we work in a classical metatheory with the axiom of choice, and we fix a single set-theoretic universe. We use the following standard terminology: 
an element of the universe is called a \emph{set}, or sometimes a \emph{small set} for emphasis, while 
a \emph{proper class} or \emph{large set} is a collection of sets which is itself not a set; 
a \emph{class} is a small or large set. 
We allow for a \emph{function} or \emph{relation} to have any class as its domain or codomain. 
We denote by $\Set$ the category of small sets and functions, corresponding to the fixed universe. Any category we consider is tacitly assumed to be \emph{locally small}, meaning that the collection of arrows between two objects is a small set.
A category is called \emph{small} if its classes of objects and morphisms are small.
We denote by $\CAT$ the 2-category of categories, functors, and natural transformations.

We recall a number of topos-theoretic facts and notations that we will need; we follow the notation of~\cite{elephant}, unless explicitly mentioned otherwise. 
By \emph{topos} we will always mean \emph{Grothendieck topos}, i.e.\ a category
equivalent to that of sheaves on a small site $(\mc C, J)$; recall
that the category $\mc C$ can always be assumed to be finitely complete. By
Giraud's theorem, a topos can be characterized intrinsically as a category $\mc
E$ such that:
\begin{enumerate}
\item $\mc E$ admits finite limits;
\item $\mc E$ admits small coproducts, which are disjoint and stable under
      pullback;
\item every equivalence relation $\rho \rightarrowtail \phi \times \phi$ in $\mc E$
    is \emph{effective}, i.e.\ a kernel pair;
\item effective epimorphisms in $\mc E$ are stable under pullback;
\item there exists a \emph{generating set} $S$ in $\mc E$, meaning that
      parallel morphisms in $\mc E$ can be separated by morphisms sourced at objects
      in $S$.
\end{enumerate}

\begin{definition}
    Given two toposes $\mc E $ and $ \mc F$, a \emph{geometric morphism} $\mc E \to
    \mc F$ is an adjoint pair
    $f_* \colon \begin{tikzcd}[cramped]
        {\mc E} & {\mc F}
        \arrow[""{name=0, anchor=center, inner sep=0}, shift right=1.3, from=1-1, to=1-2]
        \arrow[""{name=1, anchor=center, inner sep=0}, shift right=1.7, from=1-2, to=1-1]
        \arrow["\dashv"{anchor=center, rotate=-90}, draw=none, from=1, to=0]
    \end{tikzcd} \cocolon {f^*}$
    whose left
    adjoint $f^* \colon \mc F \to \mc E$ preserves finite limits; we refer to $f^*$ 
    as the \emph{inverse image}, and to its right adjoint $f_*$ as the \emph{direct
    image}.    We denote by $\mathbf{Topos}$ the 2-category of toposes, geometric morphisms, and natural transformations between their inverse images.
\end{definition}

\begin{definition}
    A \emph{point} of a topos $\mc E$ is a geometric morphism $\Set \to \mc
    E$, which we will systematically identify with its inverse image $\mc E \to
    \Set$. We denote by $\pt (\Ec)$ the category of points of $\mc E$. 
\end{definition}

Any class $X$ of points of $\Ec$ induces a canonical geometric morphism $\Set^X \to \Ec$, whose inverse image is the evaluation functor $\eval{-} \colon \Ec \to \Set^X$.

\begin{definition}
    A class $X$ of points of $\Ec$ is called \emph{separating} if the functor $\eval{-}\colon\Ec \to \Set^X$ is conservative. If such a class $X$ exists, then we say that $\Ec$ has \emph{enough points}: in that case, $\Ec$ then admits in particular a separating small \emph{set} of points.
\end{definition}

\begin{definition}
    An \emph{infinitary pretopos} (called $\infty$-pretopos in~\cite{elephant}) is a category satisfying (1)-(4) of Giraud's axioms above. A \emph{pretopos} is a category satisfying axioms (1), (3), (4), and axiom (2) restricted to only \emph{finite} coproducts. We recall that, in a pretopos, every epimorphism is
effective. 
\end{definition}

\begin{definition}
    A  \emph{pretopos morphism} (resp. \emph{infinitary-pretopos morphism}) is a
functor between pretoposes (resp. infinitary pretoposes) which preserves finite
limits, finite coproducts (resp. small coproducts), and
epimorphisms. The inverse image of a geometric morphism of toposes is in particular
an infinitary-pretopos morphism. %
\end{definition}

By the definition in terms of sheaves on a site, a topos admits a small dense full subcategory which is
closed under finite limits. We will repeatedly use this fact, together with the following three lemmas, in the proof of our main theorem.
\begin{lemma}
\label{lem:extension-from-a-dense-subcategory}
Let $F , G \colon \mc C \to \mc D$ be a pair of functors and let
$i\colon \mc S \hookrightarrow \mc C$ be the inclusion of a dense full
subcategory. If $F$ preserves all small colimits existing in $\mc C$, then every natural transformation
$F \circ i \implies G \circ i$ admits a unique extension to a natural
transformation $F \implies G$.
\end{lemma}
\begin{proof}
Follows, e.g., by \cite[Thm~5.1]{kelly2005}. 
\end{proof}

Recall that a category $\mc{I}$ is \emph{cofiltered} if every finite diagram in
$\mc{I}$ admits a cone.
\begin{lemma}
\label{lem:cofiltered-category-of-elements}
Let $\mc C$ be a finitely complete category and let $F \colon \mc C
\to \Set$ be a functor. The category of elements $\int F$ is cofiltered if and
only if $F$ preserves finite limits.
\end{lemma}
\begin{proof}
Follows immediately from, e.g., \cite[Thm.~VII.6.3 and Cor.~VII.6.4]{maclane-moerdijk}.
\end{proof}
In particular, a poset $\mc{P}$ is cofiltered if it is non-empty and for any
$p_1, p_2 \in \mc{P}$ there exists $p \in \mc{P}$ such that $p \leq p_1$ and $p
\leq p_2$. For $\mc{P}$ a cofiltered poset and $\mc{I}$ a cofiltered
category, a functor $d \colon \mc{P} \to \mc{I}$ is \emph{initial} if (i)
for every $i \in \mc I$ there exist $p \in \mc P$ and an arrow $d(p) \to i$ in $\mc I$,
and (ii) for any pair of arrows $f_1 \colon d(p_1) \to i$,
$f_2 \colon d(p_2) \to i$ in $\mc{I}$ there exists $p \in \mc{P}$ such
that $p \leq p_1$, $p \leq p_2$, and the following square commutes:
\[\begin{tikzcd}
d(p) \ar[r, "d(p \leq p_1)"] \ar[d, "d(p \leq p_2)"'] & d(p_1) \ar[d,
"f_1"] \\
d(p_2) \ar[r, "f_2"'] & i.
\end{tikzcd}\]

\begin{lemma}[\protect{\cite[Prop.~I.8.1.6]{SGA4}, \cite[Thm.~1]{AndNem1982}}]\label{lem:initial-functor-cofiltered}
For any small cofiltered category $\mc{I}$, there exist a cofiltered poset $\mc{P}$ and an initial functor $d \colon \mc{P} \to \mc I$.
\end{lemma}

\section{Ultraconvergence spaces}\label{sec:ultraconvergence-spaces}
In this section, we introduce \emph{ultraconvergence spaces}, which simultaneously generalize topological spaces and categories.

\subsection{Ultrafilters and ultrafamilies}

\paragraph{Ultrafilters.}We begin by recalling some preliminaries on ultrafilters.

\begin{definition}
An \emph{ultrafilter} on a set $I$ is a collection
$\mu$ of subsets of $I$ whose characteristic function $\cP(I) \to \two$ is a
Boolean algebra homomorphism. 
For any set $I$, we write $\beta I$ for the set of ultrafilters on $I$. For any
$i \in I$, the collection $[i] \coloneqq \{ A \subseteq I \ | \ i \in A\}$ is
an ultrafilter, called the \emph{principal ultrafilter} at $i$. The assignment
$i \mapsto [i]$ is an injective function from $I$ to $\beta I$, and we may
identify $I$ with its image in $\beta I$. Any function $f \colon I \to J$
extends to a function $f \colon \beta I \to \beta J$, which sends an
ultrafilter $\mu \in \beta I$ to $f(\mu) := \{ B \subseteq J \ | \ f^{-1}(B) \in \mu \}$. 
\end{definition}

A \emph{proper filter} on a set $I$ is a collection  of subsets of $I$
which is upward closed, closed under finite intersections, and does not contain the empty set.
The \emph{ultrafilter principle} is a weak form of the axiom of choice which states that any proper filter on $I$ is contained in some ultrafilter.
In particular, if $I$ is infinite, then there exist
infinitely many non-principal ultrafilters on $I$. 

Let $\mu$ be an ultrafilter on a set $I$. We call a subset $A$ of $I$ \emph{$\mu$-large} if
$A \in \mu$, and, more generally, when $P(i)$ is any property of elements $i
\in I$, we say that $P(i)$ \emph{holds $\mu$-eventually} if $\{i \in I \ | \
P(i)\} \in \mu$.
For any $I_0 \in \mu$, $\mu$ \emph{restricts} to an ultrafilter $\restr \mu
{I_0}$ on $I_0$, consisting of the $\mu$-large subsets of $I_0$; we have
$\iota(\restr \mu {I_0}) = \mu$, where $\iota \colon I_0 \hookrightarrow I$ is the
inclusion. We denote by $1$ the unique ultrafilter on the singleton set $\{*\}$. 

We now recall the important operation of \emph{dependent sum} of ultrafilters (see, e.g., \cite[\S 15]{Blass1970} and \cite[Def.~6]{Garner2020}). 

\begin{definition}\label{def:depsum}
Let $\mu$ be an ultrafilter on a set $I$ and, for each $i \in I$, let $\nu_i$ be an
ultrafilter on a set $J_i$. The \emph{dependent sum} $\sum_{i \to \mu} \nu_i$ is the
ultrafilter on the disjoint union $\sum_{i \in I} J_i$ consisting of those
subsets $U \subseteq \sum_{i \in I} J_i$ for which $U \cap J_i$ is
$\nu_i$-large $\mu$-eventually. 
\end{definition} 
To see that the collection $\sum_{i \to \mu} \nu_i$ defined in Definition~\ref{def:depsum} is indeed an ultrafilter, note that, if $m \colon
\mc P(I) \to \two$ and $n_i \colon \mc P(J_i) \to \two$ are the characteristic
functions of $\mu$ and $\nu_i$, respectively, then the characteristic function
of $\sum_{i \to \mu} \nu_i$ is the composite
\[ \begin{tikzcd}
{\mc P\left(\sum_{i \in I} J_i\right) \cong \prod_{i \in I} \mc P(J_i)} &[20pt] {\two^I \cong \mc P(I) } & {\two .}
\arrow["{\prod_{i \in I} n_i}", from=1-1, to=1-2]
\arrow["m", from=1-2, to=1-3]
\end{tikzcd} \]

We sometimes write $(i\to\mu) \otimes \nu_i$ in place of $\sum_{i\to\mu} \nu_i$, and in particular $\mu\otimes \nu$ if each of the ultrafilters $\nu_i$ is equal to the same ultrafilter $\nu$ on a set $J$. Note that this ultrafilter, called the \emph{tensor product} of $\mu$ and $\nu$, has index set $\sum_{i \in I} J \cong I \times J$, and it satisfies $\pi_I(\mu \otimes \nu) = \mu$ and $\pi_J(\mu \otimes \nu) = \nu$ where $\pi_I$ and $\pi_J$ are the two projections. 

\begin{remark}
    The association $I \mapsto \beta I$, extended to functions as described above, determines a functor $\beta \colon \Set \to \Set$ carrying the structure of a \emph{monad}: at a set $I$, its unit $\mathfrak u \colon \id \Rightarrow \beta$ is defined by the inclusion $I \hookrightarrow \beta I$ of principal ultrafilters, while its multiplication $\mathfrak m \colon \beta^2 \Rightarrow \beta$ is defined by mapping an ultrafilter $\Xi \in \beta^2 I$ to the family $\set{ I_0 \subseteq I | \set{ \mu \in \beta I | I_0 \in \mu} \in \Xi }$. The above direct sum can then be recovered in terms of the monad structure: for $J \coloneqq \sum_{i\in I} J_i$, the assignment $g(i) \coloneqq \iota_i(\nu_i)$ defines a map $g \colon I \to \beta J$ which satisfies $\sum_{i \to \mu} \nu_i = \mathfrak m_J( g(\mu))$.
\end{remark}

\paragraph{Ultrafamilies and ultraproducts.}\label{par:ultrafamilies}We now recall the construction of {ultraproducts} of sets. Fix an ultrafilter $\mu$ on a set $I$, and, for each $i \in I$, a class $X_i$. 

\begin{definition}
A \emph{$\mu$-family} $(x_i)_{i \to \mu}$ in $(X_i)_{i \in I}$ is the
$\mu$-equivalence class of a family $(x_i)_{i \in I_0}$, indexed over some
$\mu$-large set $I_0$, with $x_i \in X_i$ for each $i \in I_0$, where two such
families are \emph{$\mu$-equivalent} if they are equal $\mu$-eventually. In
particular, when each $X_i$ is equal to a fixed class $X$, we call this a
$\mu$-family in $X$. By an \emph{ultrafamily} we mean a $\mu$-family for some
ultrafilter $\mu$.

\looseness=-1
If each class $X_i$ is a set, then the collection of $\mu$-families in
$(X_i)_{i \in I}$ is also a set, called the \emph{ultraproduct} of the family
$(X_i)_{i \in I}$ along $\mu$, which we denote by $\prod_{i \to \mu} X_i$. 
In the case where
each $X_i$ is equal to a fixed set $X$, we write $X^\mu$ for $\prod_{i \to \mu}
X$, and call this the \emph{ultrapower} of $X$ by $\mu$. 

\end{definition}

Let $\kappa$ be an ultrafilter on a set $K$ and let $h \colon K \to I$ be a function such that $\mu = h(\kappa)$. There is a canonical function $\epsilon_h \colon \prod_{i\to\mu} X_i \to \prod_{k\to\kappa}X_{h(k)}$ which sends $(x_i)_{i \to \mu}$ to $(x_{h(k)})_{k \to \kappa}$. 
Note that, if $h$ is injective, then $\epsilon_h$ is a bijection. 
In particular, for any $i_0 \in I$, we have 
$\prod_{i\to[i_0]} X_i \cong \prod_{*\to 1} X_{i_0}$,
which identifies an $[i_0]$-family $(x_i)_{i\to[i_0]}$ with the $1$-family $(x_{i_0})_{* \to 1}$.

Ultrafamilies indexed by a dependent sum $\sum_{i \to \mu} \nu_i$, where $\nu_i$ is an ultrafilter on a set $J_i$ for each $i \in I$, can be described as follows.
For each $(i,j) \in \sum_{i \in I} J_i$, let $X_{i,j}$ be a class. 
Note that an ultrafamily $(x_{i,j})_{(i,j) \to \sum_{i \to I} \nu_i}$ can be canonically identified with an ultrafamily of ultrafamilies, $((x_{i,j})_{j \to \nu_i})_{i \to \mu}$. If all the classes $X_{i,j}$ are small, this identification yields a bijection of ultraproducts 
$$\prod_{(i,j) \to \sum_{i \to \mu} \nu_i} X_{i,j} \ \cong \ \prod_{i\to\mu}\left( \prod_{j \to \nu_i} X_{i,j} \right).$$ 
In light of this identification, as a convenient notation, we denote both of these ultraproducts by $\prod_{i \to \mu; j \to \nu_i} X_{i,j}$, and we write $(x_{i,j})_{i \to \mu; j \to \nu_i}$ for an ultrafamily in it.

\paragraph{Categories of ultrafamilies.} The \emph{category of
ultrafilters} $\UF$~\cite[\S 2]{Blass1970} has as objects pairs $(I, \mu)$ where $I$
is a set and $\mu \in \beta I$, and $\UF((I, \mu), (J,\nu))$ is defined as the
subset of $J^\mu$ consisting of those $\mu$-classes of functions $f \colon I
\to J$ for which $f(\mu) = \nu$. 

We generalize this as follows. For a class
$X$, we denote by $\bbbeta X$ the category whose objects are ultrafamilies in
$X$, and, for ultrafamilies $(x_i)_{i \to \mu}$ and $(y_j)_{j \to \nu}$, the
morphisms $(x_i)_{i \to \mu} \to (y_j)_{j \to \nu}$  are $\UF$-morphisms $f
\colon (J, \nu) \to (I, \mu)$ such that $(x_{f(j)})_{j \to \nu} = (y_j)_{j \to
\nu}$. In particular, $\bbbeta 1$ coincides with $\UF^\op$.

\begin{remark}
The notation $\bbbeta X$ follows~\cite[Def.~1.18]{saadia2025}. The category
$(\bbbeta X)^\op$ is equivalent to the category $\UF_{\mc{C}}$ introduced
in~\cite[Def.~21]{Garner2020}, where $\mc{C}$ is the category of small
presheaves on the discrete category $X$.
The assignment $X \mapsto \bbbeta X$ extends to a pseudomonad on $\CAT$ which was introduced by Rosolini~\cite{rosolini-ultracompletion} and called the \emph{ultracompletion} pseudomonad.
\end{remark}

We highlight here a property of $\UF$ that we will use in the following: namely, the fact that it satisfies the \emph{right Ore condition}.

\begin{lemma}[\protect{\cite[Prop.~39]{Garner2020}}]\label{lem:completing-cospans}
    Any cospan $(J, \nu) \to (I,\mu) \leftarrow (K,\kappa)$ in $\UF$ can be completed to a commuting square. 
\end{lemma}

\begin{remark}\label{rem:functoriality-tensor-product}
As noted in \cite[Def.~6]{Garner2020}, the tensor product $\otimes$ gives a
monoidal structure on $\cat{UF}$. Dependent sums are also functorial, in a sense that we do not need to make precise for the sake of this paper (see, e.g., \cite[Def.~1.22]{saadia2025}). However,
we will need the following two constructions, which may be seen as arising from the functoriality of the dependent sum:
\begin{enumerate}
    \item every $\UF$-arrow $h \colon (K , \kappa) \to (I, \mu)$ induces, for any $\mu$-family $( (J_i, \nu_i) )_{i \to \mu}$ of objects of $\UF$, a 
        $\UF$-arrow $h\otimes \id$ from $(k\to\kappa) \otimes \nu_{h(k)}$ to $(i\to\mu) \otimes \nu_i$; its underlying function $\sum_{k \in K} J_{h(k)} \to \sum_{i\in I} J_i $ is given by $(k, j) \mapsto (h(k), j)$;
    \item every $\mu$-family of $\UF$-arrows $(h_i \colon (K_i, \kappa_i) \to (J_i,\nu_i))_{i\to\mu}$ induces a $\UF$-arrow $\id \otimes (h_i)_{i \to \mu}$ from $(i\to\mu) \otimes \kappa_i$ to $(i\to\mu) \otimes \nu_i$; its underlying function $\sum_{i\in I} K_i \to \sum_{i\in I} J_i$ is given by $(i,j) \mapsto (i, h_i(j))$.
\end{enumerate}
\end{remark}

\subsection{Ultraconvergence spaces}

\paragraph{Definition of ultraconvergence spaces.}
We now define the generalization of topological spaces, viewed as relational
$\beta$-modules as explained in Section \ref{introduction}, to the categorical level, replacing the $\two$-valued
convergence relation by a $\Set$-valued relation. We propose to call this
concept an \emph{ultraconvergence space} to emphasize the topological
intuition. We will first give a concrete definition; see Remark~\ref{rem:on-uc-def} for a more abstract viewpoint. The same concept was axiomatized independently
in~\cite{saadia2025}, where it is called \emph{virtual
ultracategory}, and a similar but slightly weaker concept, called \emph{generalized ultracategory}, appears in~\cite{hamad2025} (in the terminology of Definition~\ref{dfn:uc-space} below, the left-naturality axiom is missing).
\begin{definition}\label{dfn:uc-space}
An \emph{ultraconvergence space} consists of the following data:
\begin{itemize}
\item a class $X$, whose elements are called \emph{points}; 
\item for every point $x \in X$ and every $\mu$-family of points $(y_i)_{i \to \mu}$, a small set denoted by $\Hom_{\ult}(x, (y_i)_{i \to \mu})$, whose elements 
are called \emph{ultra-arrows} and denoted $r \colon x \ult \lim_{i \to \mu} y_i$;
\item for every $x \in X$, an \emph{identity} ultra-arrow $\id_x \colon x \ult \lim_{\ast \to 1} x$;
\item for every $\UF$-arrow $h \colon (J, \nu) \to (I, \mu)$ and every ultra-arrow $r \colon x \ult \lim_{i \to \mu} y_i$, a \emph{reindexing} $r[h] \colon x \ult \lim_{j \to \nu} y_{h(j)}$;
\item for every ultra-arrow $r \colon x \ult \lim_{i \to \mu} y_i$ and every $\mu$-family of ultra-arrows $s_i \colon y_i \ult \lim_{j \to \nu_i} z_{i,j}$, a \emph{composition} $(s_i)_{i \to \mu} \cdot r \colon x \ult \lim_{i \to \mu; j \to \nu_i} z_{i,j}$,
\end{itemize}
\vbox{subject to the following axioms:
\begin{enumerate}
    \item \label{it:homfunctor} (functoriality) $r[\id_{(I,\mu)}] = r$ and $r[h \circ h'] = r[h][h']$;
    \item \label{it:natural1} (left-naturality) $(s_{h(k)})_{k\to\kappa} \cdot r[h] = ( (s_i)_{i\to\mu} \cdot r )[h \otimes \id]$;
    \item \label{it:natural2} (right-naturality) $\left(s_i[h_i]\right)_{i\to\mu} \cdot r = ( (s_i)_{i\to\mu}\cdot r) \left[ \id \otimes (h_i)_{i \to \mu} \right]$;
    \item \label{it:identityR} (right-identity) $(r)_{*\to 1} \cdot \id_x = r$;
    \item \label{it:identityL} (left-identity) $(\id_{y_i})_{i\to\mu} \cdot r = r$;
    \item \label{it:associativity} (associativity) $(t_{i,j})_{i\to \mu; j \to \nu_i} \cdot ( (s_i)_{i\to\mu} \cdot r) = ( (t_{i,j})_{j\to\nu_i} \cdot s_i )_{i\to\mu} \cdot r$;
\end{enumerate}
for any $r$, $(s_i)_{i \to \mu}$, $h$ as above, any ultra-family of ultra-arrows $(t_{i,j})_{i \to \mu; j \to \nu_i}$, any $\UF$-arrow $h' \colon (K,\kappa) \to (J,\nu)$, and any $\mu$-family of $\UF$-arrows $(h_i \colon (K_i, \kappa_i) \to (J_i, \nu_i))_{i \to \mu}$.}

\looseness=-1 
We will refer to the data $\braket{\Hom_{\ult}, \id, \cdot}$ satisfying
the above axioms as an \emph{ultraconvergence structure} on the class $X$. 
An ultraconvergence space is \emph{small} if its class of points is a small set.
\end{definition}

\begin{remark}\label{rem:on-uc-def}
    To justify the first three axioms in Definition~\ref{dfn:uc-space}, and explain the chosen terminology, note first that axiom~(\ref{it:homfunctor}) precisely says that $\Hom_{\ult} \colon X \times \bbbeta X \to \cSet$ is a functor, where we regard $X$ as a discrete category and, for any $x \in X$ and $\bbbeta X$-arrow $h \colon (y_i)_{i \to \mu} \to (y_{h(j)})_{j \to \nu}$, we let $\Hom_{\ult}(h) \colon \Hom_{\ult}(x, (y_i)_{i \to \mu}) \to \Hom_{\ult}(x,(y_{h(j)})_{j \to \nu})$ be the function that maps $r$ to its reindexing $r[h]$. Moreover, the composition operation is a function  
\[\prod_{i\to\mu} \Hom_{\ult} ( y_i , (z_{i,j})_{j\to \nu_i} )  \times \Hom_{\ult} ( x , (y_i)_{i\to\mu} ) \to \Hom_{\ult} ( x , (z_{i,j})_{i\to \mu; j \to \nu_i})\]
and axioms~(\ref{it:natural1})~and~(\ref{it:natural2}) express suitable naturality properties of this operation. 
\end{remark}

\begin{remark}
    In other words, \Cref{rem:on-uc-def} expresses that ultraconvergence arrows determine a \emph{profunctor} $\Hom_{\ult} \colon  \bbbeta X \pro X$, playing the role of relations between categories. In formal analogy with topological spaces, proved in \cite{barr} to be  algebras for an extension of the ultrafilter monad $\beta$ to relations, the definition of an ultraconvergence space can be repackaged into that of an algebra for an extension of the ultracompletion pseudomonad $\bbbeta$ to profunctors: this is proved by the third author, together with Quentin Aristote, in the recent work \cite{aristote2026profunctorialalgebras}.
\end{remark}

\paragraph{Examples of ultraconvergence spaces.} 
Before giving the natural notion of structure-preserving map between ultraconvergence spaces (Definition~\ref{def:continuous-maps} below), we give a number of examples, showing in particular how ultraconvergence spaces generalize both categories and topological spaces at the object level. We will see in Remarks~\ref{rem:sp-and-alex}~and~\ref{rem:top-into-ultsp} below that these constructions are indeed embeddings also at the categorical level.

\begin{constr}\label{con:alexandroff}
    Let $\mc{C}$ be a category. The \emph{Alexandroff space} $\Alex{(\mc{C})}$ of $\mc{C}$ is the ultraconvergence space whose points are the objects of $\mc{C}$, with $\Hom_{\ult}(x, (y_i)_{i \to \mu}) \coloneqq \prod_{i \to \mu} {\mc{C}}(x, y_i)$. The identity and composition of $\Alex{(\mc{C})}$ are derived from those in $\mc{C}$, and the reindexing of an ultra-arrow $r = (r_i)_{i \to \mu} \in \prod_{i \to \mu} {\mc{C}}(x, y_i)$ along $h \colon (J, \nu) \to (I, \mu)$ is defined to be $(r_{h(j)})_{j \to \nu}$.
\end{constr}
By analogy with the notion of \emph{Alexandroff topological space}, we call an \emph{Alexandroff ultraconvergence space} any ultraconvergence space of the form $\Alex{(\mc{C})}$, for $\mc{C}$ a category.
It is possible to recover the category $\mc{C}$ from the associated Alexandroff ultraconvergence structure, in the same way that a preorder can be recovered from its associated Alexandroff topology. In fact, any ultraconvergence space has a naturally associated category, as follows.
\begin{constr}\label{con:spec-cat}
    Let $X$ be an ultraconvergence space. The \emph{specialization category} of $X$ is the category $\Sp{(X)}$ whose objects are the points of $X$, and whose morphisms $x \to y$ are the ultra-arrows $x \ult \lim_{*\to 1} y$. Note that axioms (\ref{it:identityR}), (\ref{it:identityL}), and (\ref{it:associativity}) in Definition~\ref{dfn:uc-space} imply that $\Sp{(X)}$ is a category. Also note that any category $\mc{C}$ is isomorphic to $\Sp{(\Alex{(\mc{C})})}$. As a %
    convention, when no confusion can arise, we use the same notation for an ultraconvergence space and its specialization category.
\end{constr}
\begin{constr}\label{constr:top-sp-as-ultraconv-sp}
    Let $X$ be a topological space.  For any point $x\in X$ and any ultrafamily $(y_i)_{i\to\mu}$ in $X$, let us write $x \preccurlyeq \lim_{i \to \mu} y_i$ if, and only if, for every open neighborhood $U$ of $x$, we have $y_i \in U$ $\mu$-eventually. Set 
$\UltArr ( x, (y_i)_{i\to\mu}) \coloneqq \{*\}$
    if $x \preccurlyeq \lim_{i \to \mu} y_i$ and 
$\UltArr ( x, (y_i)_{i\to\mu}) \coloneqq \emptyset$
otherwise. This is a two-valued ultraconvergence structure on $X$. 
\end{constr}

\begin{remark}\label{rem:top-sp-from-ultraconv-structure}
    Let $X$ be a topological space, seen as a two-valued ultraconvergence space by Construction~\ref{constr:top-sp-as-ultraconv-sp}. We show how to recover the topology on $X$ from its ultraconvergence structure. The functor $\Hom_{\ult} \colon X \times \bbbeta X \to \two$ restricts to a relation ${\preccurlyeq} \subseteq X \times \beta X$, where we write $x \preccurlyeq \mu$ for $x \preccurlyeq \lim_{i \to \mu} i$; explicitly, this means that  every open neighborhood of $x$ is $\mu$-large. One may then show that a subset $U \subseteq X$ is open if, and only if, if $x \in U$ and $(x_i)_{i \to \mu}$ is a $\mu$-family such that $x \preccurlyeq \lim_{i \to \mu} x_i$, then $x_i \in U$ holds $\mu$-eventually. In particular, we recover the \emph{specialization order} on $X$ as the (thin) category $\Sp{(X)}$: $x \preccurlyeq y$ if, and only if, every open neighborhood of $x$ contains $y$.    
A small ultraconvergence space $X$ corresponds to a topological space in this way if and only if there is always at most one ultra-arrow of a given type, and if moreover for every $\UF$-arrow $h\colon (J,\nu)\to (I,\mu)$, if $x \preccurlyeq \lim_{j\to\nu} y_{h(j)}$, then $x \preccurlyeq \lim_{i\to\mu} y_i$. This precisely recovers the Galois connection between relational $\beta$-module structure and topologies on $X$ of \cite{barr}. 
\end{remark}

We now give two crucial examples of ultraconvergence spaces that are in general \emph{not} obtained by the above constructions.

\begin{example}\label{ex:ultraconvergence-space-set}
We define an ultraconvergence structure on the class of sets. 
For a set $A$ and a $\mu$-family $(B_i)_{i\to \mu}$ of sets, 
we define $\Hom_{\ult}(A, (B_i)_{i \to \mu})$
to be the set of functions from $A$ into
    the ultraproduct $\prod_{i\to \mu} B_i$. The identity ultra-arrow $\id_A$ sends $a \in A$ to $(a)_{\ast \to 1} \in \prod_{\ast \to 1} A$.
For an ultra-arrow $r \colon A
\to\prod_{i\to \mu} B_i$ and a $\UF$-arrow $h \colon (K,\kappa) \to(I,\mu)$,
the reindexing $r[h]$ is given by post-composition with the induced function $\epsilon_h$:
\[ \begin{tikzcd}
A & {\prod_{i\to\mu} B_i} & {\prod_{k\to\kappa} B_{h(k)}}
\arrow["r", from=1-1, to=1-2]
\arrow["\epsilon_h", from=1-2, to=1-3]
\end{tikzcd} \]
while, for a $\mu$-family of ultra-arrows $s_i \colon B_i \to\prod_{j\to
\nu_i} C_{i,j}$, the composite $(s_i)_{i\to \mu}\cdot r$ is given by
\[ \begin{tikzcd}
A & {\prod_{i\to\mu} B_i} &[20pt] {\prod_{i\to\mu} \left(
    \prod_{j\to\nu_i} C_{i,j} \right).}
\arrow["r", from=1-1, to=1-2]
\arrow["{{\prod_{i\to\mu} s_i}}", from=1-2, to=1-3]
\end{tikzcd} \]
    We denote by $\Set$ the ultraconvergence space so defined. The specialization category of this ultraconvergence space is the category of sets and functions, also denoted $\Set$. 
\end{example}

We now show how, more generally, any geometric theory $\mathbb{T}$ (see, e.g., \cite[Sec.~D1.1]{elephant}) gives rise to an ultraconvergence structure on its class of models. 
\begin{example}\label{ex:ultraconvergence-space-modT}
    \looseness=-1 
Let $\mathbb{T}$ be a geometric theory in a signature $\Sigma$. 
    We define an ultraconvergence structure on the class $\Mod(\mathbb{T})$ of ($\Set$-)models of $\mathbb{T}$.
For any model $M$ of $\mathbb{T}$ and any $\mu$-family
$(N_i)_{i\to\mu}$ of models of $\mathbb{T}$, 
define $\Hom_{\ult}(M, (N_i)_{i \to \mu})$ to be the set of 
homomorphisms from the $\Sigma$-structure $M$ to the ultraproduct
    $\Sigma$-structure $\prod_{i\to\mu}N_i$. Note that $\prod_{i \to \mu} N_i$ need not be a model of $\mathbb{T}$ in general, but it is still a $\Sigma$-structure. The identity, composition, and reindexing structure on $\Mod(\mathbb{T})$ are defined in the same way as in Example~\ref{ex:ultraconvergence-space-set}, noting that these operations indeed all yield homomorphisms of $\Sigma$-structures.
The specialization category of this ultraconvergence space is 
the category of $\mathbb T$-models and homomorphisms, also denoted $\Mod(\mathbb{T})$.

\end{example}
Let us make a few remarks about the above two examples, also relating them to the literature on ultracategories. First note that Example~\ref{ex:ultraconvergence-space-set} is in fact a special case of Example~\ref{ex:ultraconvergence-space-modT}, namely, the ultraconvergence space $\Set$ is $\Mod(\mathbb{O})$, where $\mathbb{O}$ is the \emph{theory of objects}, i.e.\ the empty theory over the signature with a single sort and no symbols except equality \cite[Sec.~D3.2]{elephant}; we will return to this in Example~\ref{ex:alexandroff-presheaf-type}.

Also note that the ultraconvergence structure on $\Set$ given in Example~\ref{ex:ultraconvergence-space-set} is \emph{not} the same as the one of the Alexandroff space $\Alex{(\Set)}$ defined in Construction~\ref{con:alexandroff}.

\begin{example}
Ultra-arrows from $A$ to $(B_i)_{i \to \mu}$ in $\Alex{(\Set)}$ are $\mu$-families of functions $(A \to B_i)_{i \to \mu}$, and the set of such $\mu$-families is in general larger than the set of functions $A \to \prod_{i \to \mu} B_i$. To see this, take $\mu$ a non-principal ultrafilter on an infinite set $I$, and $B_i = \two$ for every $i \in I$. Then, since $\prod_{i \to \mu} B_i = \two$, the $\Set$-ultra-arrows from any set $A$ to $(B_i)_{i \to \mu}$ are in bijection with $\mathcal{P}(A)$, but the $\Alex{(\Set)}$-ultra-arrows from $A$ to $(B_i)_{i \to \mu}$ are in bijection with $\mathcal{P}(A)^\mu$. Thus, $\Alex{(\Set)}$ is not isomorphic to the space $\Set$ of Example~\ref{ex:ultraconvergence-space-set}.
\end{example}

Next recall that, for any geometric theory $\mathbb{T}$, the class of models $\Mod(\mathbb{T})$ coincides with the class of points of the classifying topos $\mathcal{E}_{\mathbb{T}}$ of $\mathbb{T}$. The ultraconvergence structure on $\Mod(\mathbb{T})$ is then what we call the \emph{canonical ultraconvergence structure} on $\pt (\mathcal{E}_{\mathbb{T}})$, see the beginning of Section~\ref{sec:main} below.

In the special case of Example~\ref{ex:ultraconvergence-space-modT} where the theory $\mathbb{T}$ is coherent, {\L}o{\' s}'s theorem implies that any ultraproduct of models of $\mathbb{T}$ is again a model of $\mathbb{T}$. In that case, ultra-arrows in $\Mod(\mathbb{T})$ are just certain morphisms in the category
$\Mod(\mathbb{T})$. For readers familiar with the theory of \emph{ultracategories} \cite{makkai, lurie}, we note that the ultraconvergence space $\Mod(\mathbb{T})$, for $\mathbb{T}$ coherent, is a special case of the ultraconvergence space that may be associated with any \emph{strong ultracategory} in the sense of \cite[Rem.~3.5]{saadia2025}: as pointed out in \cite[Exa.~4.2(vii)]{saadia2025}, given a strong ultracategory $\mathcal{M}$,  we can equip it with an ultraconvergence structure by setting $\Hom_{\ult}(M, (N_i)_{i \to \mu}) \coloneqq \mathcal{M}\left(M, \int_{I} N_i \, d\mu\right)$. For $\mathbb{T}$ a coherent theory, $\mathcal{M} = \Mod(\mathbb{T})$ is a strong ultracategory, and this is then the ultraconvergence structure of Example~\ref{ex:ultraconvergence-space-modT}.

\begin{remark}
Note that \Cref{dfn:uc-space} presents evident similarities with Clementino and Tholen's notion of ultracategory \cite[\S 11]{clementino03}. The crucial difference lies in our \emph{reindexings}: indeed, an ultracategory $X$ in the sense of \cite{clementino03} can be seen as a (small) ultraconvergence space in which the reindexing map $\Hom_{\ult}(h) \colon \Hom_{\ult}(x, (y_i)_{i \to \mu}) \to \Hom_{\ult}(x,(y_{h(j)})_{j \to \nu})$, for any $\bbbeta X$-arrow $h \colon (y_i)_{i \to \mu} \to (y_{h(j)})_{j \to \nu}$, is a bijection. This condition, however, is not satisfied by our main example of ultraconvergence spaces (\Cref{ex:ultraconvergence-space-modT}): a priori we can guarantee $\Hom_{\ult}(h)$ to be a bijection only if $h$ is injective (and hence an isomorphism in $\UF$).
\end{remark}

\subsection{The 2-category of ultraconvergence spaces}

We now introduce the natural notion of morphism between ultraconvergence spaces as maps preserving ultraconvergence structures. 
\begin{definition}\label{def:continuous-maps}
Let $X$ and $Y$ be ultraconvergence spaces. A \emph{continuous map} $f \colon X \to Y$ is a function from the points of $X$ to the points of $Y$, together with, for every $x\in X$ and every ultrafamily $(x_i)_{i\to\mu}$ in
$X$, a function
   $ \Hom_{\ult}^X \left(x , (x_i)_{i\to \mu}\right) \to
\Hom_\ult^Y \left( f(x) , ( f(x_i) )_{i\to \mu} \right)$
which preserves identities, reindexings, and compositions:
\begin{enumerate}
\item $f ( \id_x ) = \id_{f(x)}$;
\item $f(r[h]) = f(r)[h]$;
\item $f ( (s_i)_{i\to\mu} \cdot r) = (f(s_i))_{i\to\mu} \cdot f(r)$.
\end{enumerate}
We will also refer to the function on ultra-arrows as a \emph{continuity structure} on the function from points of $X$ to points of $Y$. 

Let $f, f' \colon X \to Y$ be continuous maps. A \emph{morphism} $\alpha
\colon f \implies f'$ is the datum of an ultra-arrow $\alpha_x \colon f(x) \ult
\lim_{* \to 1} f'(x)$ for every $x \in X$, such that $ f'(r) \cdot \alpha_x =
(\alpha_{x_i})_{i\to \mu} \cdot f(r)$ for each ultra-arrow $r \colon x \ult
\lim_{i\to \mu} x_i$ in $X$.

Ultraconvergence spaces, continuous maps and their morphisms determine a
(strict) 2-category, which we denote by $\UltSp$.
\end{definition}

Building on Constructions~\ref{con:alexandroff},~\ref{con:spec-cat}, and~\ref{constr:top-sp-as-ultraconv-sp}, we show how the 2-category of ultraconvergence spaces is a joint generalization of both topological spaces and categories.
\begin{remark}\label{rem:sp-and-alex}
    Every continuous map $f \colon X \to Y$ of ultraconvergence spaces determines a functor $\Sp (f) \colon \Sp(X) \to \Sp(Y)$, and the components of a morphism of continuous maps $\alpha \colon f \implies f'$ determine a natural transformation $\Sp (\alpha) \colon \Sp (f) \implies \Sp (f')$. Therefore, the specialization category construction yields a 2-functor $\Sp(-) \colon \UltSp \to \CAT$. 

      Conversely, every functor $F \colon \mc C \to \mc D$ determines a continuous map $\Alex (F) \colon \Alex ({\mc C}) \to \Alex(\mc D)$ by mapping an ultra-arrow $(f_i \colon c \to c_i)_{i\to\mu}$ in $\mc C$ to the ultra-arrow $(F(f_i) \colon F(c) \to F(c_i))_{i\to\mu}$ in $\mc D$, and the components of a natural transformation $\alpha \colon F \implies F'$ determine a morphism of continuous maps $\Alex (\alpha) \colon \Alex(F) \implies \Alex(F')$. Therefore, the Alexandroff space construction yields a 2-functor $\Alex{(-)} \colon \CAT \to \UltSp$. 

      Moreover, note that $\Alex{(-)}$ is left 2-adjoint to $\Sp{(-)}$. Since, for any category $\mc{C}$, the unit $\mc C \cong \Sp{(\Alex (\mc C))}$ is an isomorphism of categories, the left adjoint $\Alex{(-)} \colon \CAT \to \UltSp$ is \emph{2-fully-faithful}, meaning that it is locally an equivalence.
\end{remark}

\begin{remark}\label{rem:top-into-ultsp}
      Let $\mathbf{Top}$ be the (locally posetal) 2-category of topological spaces, where for two continuous maps $f,g \colon X \to Y$ we write $f \leq g$ if $f(x) \preccurlyeq g(x)$ holds in the specialization order of $Y$ for each $x \in X$. Then $\mathbf{Top}$ embeds 2-fully-faithfully in $\mathbf{UltSp}$ via Construction \ref{constr:top-sp-as-ultraconv-sp}. 
    For general ultraconvergence spaces, continuity is \emph{additional structure} on the function from points to points. However, in the special case of a function $f \colon X \to Y$ where $Y$ is the ultraconvergence space associated to a topological space, there is at most one such continuity structure. The $2$-fully-faithfulness of the embedding $\mathbf{Top}$ into $\mathbf{UltSp}$ means in particular that, for a function $f \colon X \to Y$ between topological spaces, this continuity structure exists precisely when $f$ is continuous in the usual sense.
\end{remark}

In the remainder of this section, we will describe pullbacks, open and closed subspaces, and closures in ultraconvergence spaces. These notions will be used in the proofs in the subsequent sections.
\begin{constr}\label{constr:pullbacks}
We describe how \emph{pullbacks} are computed in
$\UltSp$. Let $f \colon Y \to X, g \colon Z \to X$ be continuous maps of
ultraconvergence spaces; their pullback
\[ \begin{tikzcd}
|[label={[label distance=-2mm]-45:\lrcorner}]|{Z \times_X Y } & Y \\
Z & X
\arrow["{\widetilde g}", from=1-1, to=1-2]
\arrow["{\widetilde f}"', from=1-1, to=2-1]
\arrow["f", from=1-2, to=2-2]
\arrow["g"', from=2-1, to=2-2]
\end{tikzcd} \]
is given as follows. The points of $Z \times_X Y$ are pairs of points $z \in Z$ and $y \in Y$
such that $g(z) = f(y)$, and an ultra-arrow $(z,y) \ult \lim_{i\to\mu}
(z_i,y_i)$ in $Z \times_X Y$ is a pair of an ultra-arrow $r \colon z \ult
\lim_{i\to\mu} z_i$ in $Z$ and an ultra-arrow $s \colon y \ult \lim_{i\to\mu}
y_i$ in $Y$ such that $g(r) = f(s)$. The continuous maps $\widetilde f$ and $\widetilde g$
are then the obvious projections, both on points and on ultra-arrows.
\end{constr}

\paragraph{Subspaces.}Recall by Construction \ref{constr:top-sp-as-ultraconv-sp}
that a topological space $X$ can be seen as an ultraconvergence space in a
canonical way, and its topology can be recovered from the ultraconvergence
structure: a subset $U \subseteq X$ is open if, and only if, for $x
\preccurlyeq \lim_{i\to\mu} x_i$ with $x \in U$, then $x_i \in U$ holds
$\mu$-eventually. For a general ultraconvergence space $X$, there may be more than
one ultra-arrow $x \ult \lim_{i\to\mu} x_i$. We extend the notation from the topological case by writing $x \preccurlyeq \lim_{i\to\mu} x_i$ if there exists an ultra-arrow $x \ult \lim_{i\to\mu} x_i$ in $X$. In formal analogy with the topological context, we can then give the following definitions.

\begin{definition}
Let $\braket{X, \Hom_\ult, \id,\cdot}$ be an ultraconvergence space. 
A \emph{subspace} of $X$ is a subclass $S$ of $X$ together with the
restrictions of $\Hom_\ult$, $\id$, and $\cdot$ to $S \times \bbbeta S$.
A subspace $U$ of $X$ is \emph{open} if, for any $x \in U$ and any ultrafamily $(x_i)_{i \to \mu}$ in $X$, if $x \preccurlyeq \lim_{i \to \mu} x_i$, then $x_i \in U$ holds $\mu$-eventually. A subspace $C$ of $X$ is \emph{closed} if its complement is open;
explicitly, $C$ is closed if, whenever $x \preccurlyeq \lim_{i\to\mu} x_i$ and $x_i \in C$ holds $\mu$-eventually, we have $x \in C$.

If $X$ is small, then the poset of open subspaces of $X$, ordered by inclusion, is a frame, which we denote by
$\Open(X)$.
For any $x \in X$, we denote by $\mc N_x$ the poset of open subspaces of $X$
containing $x$, which we refer to as the \emph{open neighborhoods} of $x$.
\end{definition}
\begin{constr}\label{con:closure}
    \looseness=-1 
Let $S \subseteq X$ be a subclass of an ultraconvergence space $X$. Define 
\[ \cl{S} \coloneqq \set{ x \in X \ | \ \text{there exists an ultrafamily } (s_i)_{i \to \mu} \text{ in } S \text{ such that } x \preccurlyeq \lim_{i \to \mu} s_i } \ . \]
Then $\cl{S}$ is the \emph{closure} of $S$, i.e., the smallest closed subspace of $X$ that contains $S$. 
Clearly, any closed subspace of $X$ that contains $S$ must also contain $\cl{S}$, and $S \subseteq \cl{S}$ as $s \preccurlyeq \lim_{\ast \to 1} s$.
To see that $\cl{S}$ is indeed closed, let $x \in X$ and suppose that $r \colon x \ult \lim_{k \to \kappa} x_k$ is an ultra-arrow such that  $x_k \in \cl{S}$ holds $\kappa$-eventually. Denote by $K$ the $\kappa$-large set of indices $k$ such that $x_k \in \cl{S}$. For each $k \in K$, pick an ultrafamily $(s_{i,k})_{i \to \nu_k}$ in $S$ and an ultra-arrow $r_k \colon x_k \ult \lim_{i \to \nu_k} s_{i,k}$. The composite ultra-arrow $(r_k)_{k \to \kappa} \cdot r \colon x \ult \lim_{k \to \kappa; i \to \nu_k} s_{i,k}$ shows that $x \in \cl{S}$. 
\end{constr}
\begin{remark}
Consider the Sierpiński space $\two$ seen as an ultraconvergence space,
meaning that $b \preccurlyeq \lim_{i\to\mu} b_i$ if, and only if, $b \leq b_i$
holds $\mu$-eventually. For a subspace $U$ of $X$, in analogy with the
topological case, the condition of being open can then be equivalently
expressed as the existence of a (unique) continuity structure on its
characteristic functor $\chi_U \colon X \to \two$.
\end{remark}

\section{Étale spaces}\label{sec:etale-spaces}
In this section, we introduce the crucial notion of an \emph{étale map} of ultraconvergence spaces, and we study the category $\Et(B)$ of \emph{étale spaces} over an ultraconvergence space $B$. After showing a number of general properties of these étale maps in Subsection~\ref{subsec:etale-general}, we show in Subsection~\ref{subsec:grothendieck} that the category $\Et(B)$ is equivalent to the category of continuous maps from $B$ to the ultraconvergence space $\Set$ of Example~\ref{ex:ultraconvergence-space-set}. We use this in Subsection~\ref{subsec:etale-pretopos} to establish that the category $\Et(B)$ is an infinitary pretopos.

\subsection{General properties of étale spaces}
\label{subsec:etale-general}

\begin{definition}\label{def:etale-map}
A continuous map of ultraconvergence spaces $\pi \colon E\to B$ is
\emph{étale} if:
\begin{enumerate}
\item for each $b \in B$, the fiber $\pi^{-1}(b)$ is small;
\item for each $e \in E$ and each ultra-arrow $r \colon \pi(e)
      \ult \lim_{i\to\mu} b_i$ in $B$, there is a unique \emph{lift} $\bar r \colon e
      \ult \lim_{i\to\mu} e_i$ in $E$ such that $\pi(\bar r) = r$.
\end{enumerate}
An \emph{étale space over $B$} is an étale map $\pi \colon E \to B$.
Given two étale spaces $\pi \colon E \to B$ and $\pi' \colon E' \to B$, a
\emph{morphism} $\alpha \colon \pi \to \pi'$ is a continuous map $\alpha \colon
E \to E'$ such that $\pi = \pi' \circ \alpha$. We denote by $\Et(B)$ the
category of étale spaces over $B$.
\end{definition}
\begin{notation}
Let $\pi \colon E \to B$ be étale and $e \in E$. For $r \colon \pi(e)\ult\lim_{i\to\mu} b_i$ an ultra-arrow in $B$, we denote by $r(e) = (r_i(e))_{i\to\mu}$ the target ultra-family of the unique lift $\bar{r}$ of $r$ with domain $e$. 
\end{notation}
\begin{remark}\label{rem:local-homeo-is-etale}
A {local homeomorphism} $\pi \colon E \to B$ between topological spaces,
seen as ultraconvergence spaces in the canonical way, is an étale map in the
sense of Definition \ref{def:etale-map}. 
    As we will see in Remark~\ref{rem:etale-maps-of-top-spaces}, the converse is also true: for any topological space $B$, the étale spaces over the ultraconvergence space associated to $B$ (Construction~\ref{constr:top-sp-as-ultraconv-sp}) are the étale spaces over $B$ in the usual sense.
This may be compared with the characterization of local homeomorphisms 
between topological spaces in terms of ultrafilter convergence 
of~\cite[Thm.~II]{homeo-via-ultra-convergence}. 
Note that, different from that work's definition of \emph{discrete fibration} in terms of ultrafilter convergence (\cite[Sec.~2]{homeo-via-ultra-convergence}), in our definition of étale map, 
ultrafamilies of points are lifted, but the indexing ultrafilter of the lifted family stays the same.
\end{remark}

The following three lemmas, Lemma~\ref{lem:etale-open}, \ref{lem:bijective-étale-is-isomorphism}, \ref{lem:etale-maps-and-pullbacks}, will be used to prove our main result. They generalize analogous properties of local
homeomorphisms between topological spaces to étale maps of ultraconvergence spaces.

\begin{lemma}\label{lem:etale-open}
Let $\pi \colon E \to B$ be an étale map between ultraconvergence spaces. Then the direct image under $\pi$ of any open subspace of $E$ is open.
\end{lemma}
\begin{proof}
Let $V$ be an open subspace of $E$.
    Let $b \in \pi[V]$ be arbitrary, and let $r \colon b \ult \lim_{i \to \mu} b_i$ be an ultra-arrow. Pick $e \in V$ such that $\pi(e) = b$.
Since $\pi$ is étale, there exists a unique ultra-arrow $\bar r \colon e \ult \lim_{i \to \mu}
e_i$ in $E$ such that $\pi(\bar r) = r$. Since $V$ is open, we
    have $e_i \in V$ $\mu$-eventually. Since $(\pi(e_i))_{i \to \mu} = (b_i)_{i \to \mu}$, we conclude that $b_i \in \pi[V]$ $\mu$-eventually. 
\end{proof}
\begin{lemma}\label{lem:bijective-étale-is-isomorphism}
    \looseness=-1 
Let $\pi \colon E \to B$ be an étale map between ultraconvergence spaces.
If $\pi$ is bijective, then it is an isomorphism in $\UltSp$.
\begin{proof}
Let $\sigma \colon B \to E$ be the inverse function of $\pi$. First,
we endow $\sigma$ with a continuity structure. Let $r \colon b \ult
\lim_{i\to\mu} b_i$ be an ultra-arrow in $B$. As $b = \pi(\sigma(b))$ and $\pi$ is
étale, there exists a unique ultra-arrow $\bar r \colon \sigma(b) \ult
\lim_{i\to\mu} e_i$ in $E$ such that $\pi(\bar r) = r$. In particular, this means
that $(\pi(e_i))_{i\to\mu} = (b_i)_{i\to\mu}$, and hence that $(e_i)_{i\to\mu} =
(\sigma(b_i))_{i\to\mu}$; thus, we can define $\sigma(r) \coloneqq \bar r$. With this definition, $\sigma$ is evidently a
continuous map, and by the uniqueness of lifts with respect to $\pi$ it follows
that $\pi(\sigma(r)) = r$ and $\sigma ( \pi ( s )) = s$ for all ultra-arrows $r$ in $B$ and
$s$ in $E$.
\end{proof}
\end{lemma}

\begin{remark}
	By analogy with the topological setting, one might be tempted to phrase Lemma~\ref{lem:etale-open} as saying that `étale maps are open maps', but we refrain from using this terminology since `open map' already has a different, stronger, meaning in topos theory~\cite[Sec.~C.3.1]{elephant}. For instance, let $\mc A$ be the ``one-arrow category'' $\{ \cdot \to \cdot\}$ and let $\mc B$ be the ``two parallel arrows category'' $\{ \cdot \rightrightarrows \cdot\}$. Let $i\colon \mc A\to \mc B$ be any of the two canonical inclusions. Seen as a continuous map $\Alex (\mc A) \to \Alex (\mc B)$, $i$ sends open subspaces to open subspaces; however, the associated geometric morphism $[\mc A,\cSet]\to[\mc B,\cSet]$ is not open in the topos-theoretic sense.%
\end{remark}

\begin{lemma}\label{lem:etale-maps-and-pullbacks}
Étale maps are stable under pullback along any continuous map.
\begin{proof}
Let $\pi \colon E \to B$ be an étale map and let $f \colon Y
\to B$ be a continuous map. Consider the pullback
\[ \begin{tikzcd}
|[label={[label distance=-2mm]-45:\lrcorner}]|{Y\times_B E} & E \\
{Y} & B
\arrow["{\widetilde f}", from=1-1, to=1-2]
\arrow["{\widetilde \pi}"', from=1-1, to=2-1]
\arrow["\pi", from=1-2, to=2-2]
\arrow["f"', from=2-1, to=2-2]
\end{tikzcd} \]
in $\UltSp$, described as in Construction \ref{constr:pullbacks}.
Towards showing that $\widetilde\pi$ is étale, note first that the fiber of $\widetilde
\pi$ at $y\in Y$ is small as it is in bijection with the set $\pi^{-1}(f(y))$.
Let now $p = (y,e) \in Y \times_B E$ and consider an ultra-arrow $r \colon y \ult
\lim_{i\to\mu}y_i$ in $Y$. By continuity of $f$ we have an ultra-arrow $f(r)
\colon f(y) \ult \lim_{i\to\mu} f(y_i)$ in $B$, where by definition $f(y) =
\pi(e)$. Therefore, as $\pi$ is étale, there exists a unique ultra-arrow $\bar
r \colon e \ult \lim_{i\to\mu} e_i$ in $E$ such that $\pi(\bar r) = f(r)$. The
pair $(r,\bar r )$ is therefore the unique ultra-arrow $(y,e) \ult
\lim_{i\to\mu} (y_i, e_i)$ lifting $r$ to $Y \times_B E$.
\end{proof}
\end{lemma}

Recall that a local homeomorphism of topological spaces is \emph{locally injective}, i.e.\
any point in its domain has an open neighborhood on which the restriction is injective. 
This is no longer generally true for étale maps of ultraconvergence spaces. 
In Proposition~\ref{prop:parallel-lifts-and-local-injectivity}, we give two  
equivalent properties for étale maps which generalize local injectivity. 
Crucially, to prove their equivalence we will have to consider \emph{small}
ultraconvergence spaces: this ensures that the set of neighborhoods of a point is small, so that we can consider ultrafilters on it. Note that, if $\pi \colon E \to B$ is étale and $B$ is small,
then $E$ is necessarily small, as its set of points is a small disjoint union of small sets.
\begin{proposition}\label{prop:parallel-lifts-and-local-injectivity}
Let $\pi \colon E \to B$ be an étale map between small ultraconvergence
spaces and let $e \in E$. The following are equivalent:
\begin{enumerate}
    \item \label{it:loc-inj} there exists an open subset $V \subseteq E$ such that $e \in V$ and $\restr{\pi}{V}$ is injective;
    \item \label{it:par-lift} for any parallel pair of ultra-arrows $r,r' \colon \pi(e) \uult \lim_{i\to\mu} b_i$ in $B$, we have $r(e) = r'(e)$.
\end{enumerate}
\end{proposition}
\begin{proof}
    (\ref{it:loc-inj}) $\To$ (\ref{it:par-lift}). Pick $V$ as in (\ref{it:loc-inj}), and let 
    $r,r' \colon \pi(e) \uult \lim_{i\to\mu} b_i$ be a parallel pair of ultra-arrows in $B$. Since $V$ is 
    open, the target ultrafamilies $(r_i(e))_{i \to \mu}$ and $(r'_i(e))_{i \to \mu}$ of the unique lifts of
    $r$ and $r'$, respectively, are both in $V$ $\mu$-eventually. Also, $\pi(r_i(e)) = b_i = \pi(r'_i(e))$ $\mu$-eventually. Since $\restr{\pi}{V}$ is injective, we conclude $r_i(e) = r'_i(e)$ $\mu$-eventually, i.e., $r(e) = r'(e)$.

    (\ref{it:par-lift}) $\To$ (\ref{it:loc-inj}). We prove the contrapositive. Suppose that, for each $U \in \mc N_e$, there exist $a_U \neq b_U$ in $U$ such that $\pi(a_U) = \pi(b_U)$. As the collection $\mc F \coloneqq \{{\downarrow} U \ | \ U \in \mc N_e\}$ is a proper filter on $\mc N_e$, choose an ultrafilter $\mc U$ on $\mc N_e$ which contains $\mc F$.

    We first claim that, for each $i \in \mc{U}$, the point $e$ is in $C_i \coloneqq \cl{\set{a_U \ | \ U \in i}}$ (Construction~\ref{con:closure}). Indeed, if we had $e \not\in C_i$, then $E \setminus C_i$ would be in $\mc{N}_e$, so that ${\downarrow}(E \setminus C_i)$ would be in $\mc{U}$. Now, for any $U \in i$, we have $a_U \in U \cap C_i$, so that $U \not\subseteq E \setminus C_i$, showing that $i$ is disjoint from ${\downarrow}(E \setminus C_i)$, but both would be in $\mc{U}$, which is impossible. 
    Thus, for each $i \in \mc{U}$, we may pick an  ultra-arrow $r_i \colon e \ult \lim_{j \to \nu_i} a_{f(i,j)}$, where $\nu_i$ is an ultrafilter on a set $J_i$, and $f(i,j) \in i$ for each $j \in J_i$. 

For each $i_0 \in \mc{U}$, write ${\downarrow}i_0 \coloneqq \{ i \in \mc{U} \ | \ i \subseteq i_0 \}$, and note that $\mc{G} \coloneqq \{ {\downarrow} i_0 \ | \ i_0 \in \mc{U} \}$ is a proper filter on $\mc{U}$. Choose an ultrafilter $\lambda$ on $\mc{U}$ which contains $\mc{G}$. 
Consider the unique $\UF$-arrow $! \colon (\mc{U}, \lambda) \to (\{\ast\}, 1)$ and the reindexing of $\id_e \colon e \ult \lim_{\ast \to 1} e$ along $!$, which gives an ultra-arrow $\id_e[!] \colon e \ult \lim_{i \to \lambda} e$. The composite $r_a \coloneqq (r_i)_{i \to \lambda} \cdot \id_e[!]$ is then an ultra-arrow $e \ult \lim_{i \to \lambda; j \to \nu_i} a_{f(i,j)}$.

Write $\nu$ for the dependent product  $(i \to \lambda) \otimes \nu_i$, an ultrafilter on the set $J \coloneqq \sum_{i \in \mc{U}} J_i$. We show that $f \colon (J, \nu) \to (\mc{N}_e, \mc{U})$ is an arrow in $\UF$, i.e., that $f(\nu) = \mc{U}$, or equivalently, $f^{-1}(i_0) \in \nu$ for every $i_0 \in \mc{U}$. Let $i_0 \in \mc{U}$ be arbitrary. Then, for each $i \in {\downarrow} i_0$, we have $J_i \subseteq f^{-1}(i_0)$, since, for any $j \in J_i$, $f(i,j) \in i \subseteq i_0$. By definition of $\nu$ and the fact that ${\downarrow}i_0 \in \lambda$, we have $\sum_{i \in {\downarrow}i_0} J_i \in \nu$, so that also $f^{-1}(i_0) \in \nu$.

Applying the same reasoning to the family $(b_U)_{U \to \mc{U}}$, we obtain an ultra-arrow $r_b \colon e \ult \lim_{i\to\lambda;k \to\kappa_i} b_{g(i,k)}$ for some $\UF$-arrow $g \colon (K, \kappa) \to (\mc N_e, \mc U)$. By Lemma~\ref{lem:completing-cospans}, we can then consider a commuting square
\[\begin{tikzcd}
	{(L,\xi)} & {(J,\nu)} \\
	{(K,\kappa)} & {(\mc N_e, \mc U)}
	\arrow["p", from=1-1, to=1-2]
	\arrow["q"', from=1-1, to=2-1]
	\arrow["f", from=1-2, to=2-2]
	\arrow["g"', from=2-1, to=2-2]
\end{tikzcd}\]
in $\UF$; let $h \colon (L, \xi) \to (\mc N_e, \mc U)$ be the diagonal map $f \circ p = g \circ q$. By reindexing $r_a$ and $r_b$ along $p$ and $q$ respectively, we obtain two ultra-arrows $s_a \colon e \ult \lim_{l\to \xi} a_{h(l)}$ and $s_b \colon e \ult \lim_{l\to \xi} b_{h(l)}$. Now $t \coloneqq \pi(s_a)$ and $t' \coloneqq \pi(s_b)$ are a parallel pair of ultra-arrows in $B$ from $\pi(e)$ to $(\pi(a_{h(l)}))_{l \to \xi} = (\pi(b_{h(l)}))_{l \to \xi}$. However, $t(e) = (a_{h(l)})_{l\to\xi}$ and $r'(e) = (b_{h(l)})_{l \to \xi}$, which are distinct because $a_U \neq b_U$ for all $U \in \mc{N}_e$. This proves the negation of (\ref{it:par-lift}), as required.

\end{proof}

\begin{remark}\label{rem:completing-cospans-ack}
    In an earlier draft of this paper, the proof of Proposition~\ref{prop:parallel-lifts-and-local-injectivity} used a direct translation to ultraconvergence spaces of Frayne's theorem \cite[Cor.~4.3.13]{ChaKei90}: given two structures for the same signature $M$ and $N$, the coherent theory of $M$ contains the coherent theory of $N$ if and only if there is a structure homomorphism from $M$ to an ultrapower of $N$.
    Gabriel Saadia and Errol Yuksel suggested to us to replace that argument with the much more elegant Lemma~\ref{lem:completing-cospans}.
\end{remark}

\begin{definition}
An étale map between small ultraconvergence spaces is \emph{locally injective} if it verifies the equivalent conditions of Proposition~\ref{prop:parallel-lifts-and-local-injectivity}.
\end{definition}
\begin{corollary}\label{cor:étale-is-local-homeo}
Let $\pi \colon E\to B$ be a continuous function between topological spaces. Then, $\pi$ is étale in the sense of Definition \ref{def:etale-map} if and only if it is a local homeomorphism.
\begin{proof}
    The `if' direction was already stated in Remark~\ref{rem:local-homeo-is-etale} above. Conversely, if $\pi$ is an étale map in the sense of Definition~\ref{def:etale-map}, then by Lemma \ref{lem:etale-open} it is open as a map of topological spaces. Moreover, $\pi$ trivially satisfies (\ref{it:par-lift}) in Proposition~\ref{prop:parallel-lifts-and-local-injectivity} at each $e \in E$, since 
there is at most one ultra-arrow $\pi(e_0) \ult \lim_{i\to\mu} b_i$ in $E$; thus, it is locally injective, and hence a local homeomorphism.
\end{proof}
\end{corollary}

\begin{remark}\label{rem:etale-maps-of-top-spaces}
    In fact, more is true when $B$ is a topological space. If $\pi \colon E \to B$ is an étale map of ultraconvergence spaces, then $E$ itself must also be a topological space seen as an ultraconvergence space in the canonical way, as can be seen using Remark~\ref{rem:top-sp-from-ultraconv-structure}. 
    Moreover, any such $\pi$ is a local homeomorphism by the previous corollary. Thus, in this case, the category $\Et(B)$ is the category of étale spaces over $B$ in the usual sense, which is equivalent to the topos of sheaves over $B$.
\end{remark}

\subsection{A Grothendieck construction for ultraconvergence spaces}
\label{subsec:grothendieck}
Consider a category $\mc C$. Recall that the Grothendieck construction, associating to each functor $\mc C \to \Set$ its category of elements fibered over $\mc C$, realizes an equivalence between the functor category $\Set^{\mc C}$ and the category of \emph{discrete opfibrations} over $\mc C$ having small fibers:
\[ \cat{DiscOpfib}(\mc C) \simeq \Set^{\mc C}.\]
Consider now an ultraconvergence space $\braket{B, \Hom_\ult, \id,\cdot}$. 
The category 
$\UltSp(B,\Set)$ 
of continuous maps from $B$ to the ultraconvergence space $\Set$ 
comes equipped with a forgetful functor to $\Set^{\Sp(B)}$. 
Similarly, every étale map over $B$ is in particular a discrete opfibration over $\Sp(B)$ having small fibers, so that there is also a forgetful functor $\Et(B) \to \cat{DiscOpfib}(\Sp(B))$. 
In this section, we will show that the equivalence $\cat{DiscOpfib}(\Sp(B)) \simeq \Set^{\Sp(B)}$ lifts to an equivalence between $\Et(B)$ and $\UltSp(B,\Set)$, in such a way that the diagram
\[ \begin{tikzcd}
{\Et(B)} & {\UltSp(B, \Set)} \\
    {\cat{DiscOpfib}(\Sp(B))} & {\Set^{\Sp(B)}}
\arrow[""{name=0, anchor=center, inner sep=0}, shift right=0.45em,
from=1-1, to=1-2]
\arrow[from=1-1, to=2-1]
\arrow[""{name=1, anchor=center, inner sep=0}, shift right=0.6em,
from=1-2, to=1-1]
\arrow[from=1-2, to=2-2]
\arrow[""{name=2, anchor=center, inner sep=0}, shift right=0.45em,
from=2-1, to=2-2]
\arrow[""{name=3, anchor=center, inner sep=0}, shift right=0.6em,
from=2-2, to=2-1]
\arrow["\simeq"{description}, draw=none, from=0, to=1]
\arrow["\simeq"{description}, draw=none, from=2, to=3]
\end{tikzcd}\]
commutes. This means that the datum of a continuity structure on a $B$-indexed family of sets is essentially the same as that of an ultraconvergence structure on its category of elements together with an étale continuity structure on the projection to $B$. Note that this indeed extends the usual correspondence, which we recover taking $B$ to be the Alexandroff ultraconvergence space of a category $\mc C$, in which case the two forgetful functors are equivalences.

\begin{remark}
	This correspondence is implicitly used in Makkai's original proof \cite{makkai}, where the Grothendieck construction is applied to a so-called \emph{pre-ultrafunctor} on page 131.
\end{remark}

In Constructions~\ref{con:cont-map-of-etale-space} and \ref{con:etale-space-of-elements}, we define functors $(-)^* \colon \Et(B) \leftrightarrows \UltSp(B, \Set) \colon \int(-)$,  and we then prove in Theorem~\ref{thm:etale-spaces-as-continuous-maps} that they yield an equivalence.

\begin{constr}\label{con:cont-map-of-etale-space}
    We construct a functor $(-)^* \colon \Et(B) \to \UltSp(B, \Set)$.
    For any étale map $\pi \colon E \to B$, the function $\pi^* \colon B \to \Set$ sends $b \in B$ to the fiber $\pi^*(b) \coloneqq \pi^{-1}(b)$. It can be equipped with a natural continuity structure, by sending an ultra-arrow $r \colon b \ult \lim_{i \to \mu} b_i$ in $B$ to the function $\pi^{-1}(b) \to \prod_{i \to \mu} \pi^{-1}(b_i)$ mapping $e \in \pi^{-1}(b)$ to the ultra-family $r(e)$, which, we recall, denotes the target of the unique lift of $r$ along $\pi$ at $e$. 

For any morphism $\alpha \colon \pi_1 \to \pi_2$ in $\Et(B)$, we define $\alpha^* \coloneqq (\restr{\alpha}{\pi_1^*(b)})_{b \in B}$, which is a morphism from $\pi_1^*$ to $\pi_2^*$ in $\UltSp(B, \Set)$. Indeed, for any ultra-arrow $r \colon b \ult \lim_{i \to \mu} b_i$ in $B$ and $e \in \pi_1^*(b)$, we need to show that $\pi_2^*(r) \cdot \alpha_b^*$ and $(\alpha_{b_i}^*)_{i \to \mu} \cdot \pi_1^*(r)$ send $e$ to the same ultrafamily. For this, note that the continuity structure of $\alpha$ must send the unique ultra-arrow $\overline{r}_1$ lifting $r$ along $\pi_1$ at $e$ to the unique ultra-arrow $\overline{r}_2$ lifting $r$ along $\pi_2$ at $\alpha(e)$. In particular, applying $\alpha$ to the target of $\overline{r}_1$ yields the target of $\overline{r}_2$, which is what we needed to show.
\end{constr}
\begin{constr}\label{con:etale-space-of-elements}
    We construct a functor $\int(-) \colon \UltSp(B, \Set) \to \Et(B)$.
    Let $f \colon B \to \Set$ be a continuous map of ultraconvergence spaces. We denote the class of pairs $(b, x)$, where $b \in B$ and $x \in f(b)$, by $\int f$, and we endow it with an ultraconvergence structure, as follows. An ultra-arrow $r \colon (b,x) \ult \lim_{i \to \mu} (b_i,x_i)$ is defined to be an ultra-arrow $r \colon b \ult \lim_{i \to \mu} b_i$ in $B$ such that the function $f(r) \colon f(b) \to \prod_{i \to \mu} f(b_i)$ sends $x$ to $(x_i)_{i \to \mu}$. Note that the identity and composition operations of $B$ then lift to identity and composition operations on $\int f$.
    The continuous map $\pi_f \colon \int f \to B$, given by projection onto the first coordinate, is an étale map by construction: for any $(b, x) \in \int f$ and $r \colon b \ult \lim_{i \to \mu} b_i$, the unique lift of $r$ along $\pi_f$ is the ultra-arrow $r \colon (b,x) \ult \lim_{i \to \mu} (b_i, x_i)$, where $(x_i)_{i \to \mu}$ is the image of $x \in f(b)$ under the function $f(r)$.

    For  any morphism $\alpha \colon f_1 \To f_2$ between continuous maps $f_1, f_2 \colon B \rightrightarrows \Set$, we have, for each $b \in B$, a function $\alpha_b \colon f_1(b) \to f_2(b)$. This family of functions $(\alpha_b)_{b \in B}$ assembles into a function $\int \alpha \colon \int f_1 \to \int f_2$ which satisfies $\pi_{f_2} \circ \left(\int \alpha\right) = \pi_{f_1}$. The continuity structure on $\int \alpha$ sends an ultra-arrow $r \colon (b, x) \ult \lim_{i \to \mu} (b_i, x_i)$ in $\int f_1$ to the unique ultra-arrow $r \colon (b, \alpha_b(x)) \ult \lim_{i \to \mu} (b_i, \alpha_{b_i}(x_i)$ in $\int f_2$, which indeed exists because $f_2(r)(\alpha_b(x)) = (\alpha_{b_i}(x_i))_{i \to \mu}$, using that $\alpha$ is a morphism in $\UltSp(B, \Set)$ and that $r$ is an ultra-arrow in $\int f_1$.
\end{constr}
\begin{theorem}\label{thm:etale-spaces-as-continuous-maps}
    The pair of functors $(-)^* \colon \Et(B) \leftrightarrows \UltSp(B, \Set) \colon \int(-)$ is an equivalence of categories.
\end{theorem}
\begin{proof}
We describe the components of the natural isomorphisms between the composite functors $(-)^* \circ \int(-)$ and $\int(-) \circ (-)^*$ and the respective identity functors, and we omit the routine proof that these are natural.

Recall that, as in the usual Grothendieck construction, for any étale map $\pi \colon E \to B$, we have a bijective function
$E \cong \int \pi^*$ between the classes of points, sending $e \in E$ to $(\pi(e), e) \in \int \pi^*$. This bijective function is moreover equipped with a natural continuity structure, sending an ultra-arrow $r \colon e \ult \lim_{i \to \mu} e_i$ to $\pi(r)$, which we note is indeed an ultra-arrow $(\pi(e),e) \ult \lim_{i \to \mu} (\pi(e_i),e_i)$ in $\int \pi^*$.

In the other direction, let $f \colon B \to \Set$ be a continuous map.
For each $b \in B$, we have a bijective function $\alpha_b \colon f(b) \cong \pi_f^*(b)$, which sends $x \in f(b)$ to the point $(b,x) \in \pi_f^{-1}(b)$. Let us show that the family of bijections $(\alpha_b)_{b \in B}$ is a morphism $f \To \pi_f^*$. For any ultra-arrow $r \colon b \ult \lim_{i \to \mu} b_i$ in $B$ and $x \in f(b)$, write $(x_i)_{i \to \mu}$ for the ultrafamily $f(r)(x)$. Note that $(\pi_f^*(r) \cdot \alpha_b)(x) = ((b_i,x_i))_{i \to \mu}$, which is equal to $((\alpha_{b_i})_{i \to \mu} \cdot f(r))(x)$, as required.
\end{proof}

\subsection{The infinitary pretopos of étale spaces over an ultraconvergence space}
\label{subsec:etale-pretopos}

Let $X$ be a topological space seen as an ultraconvergence space in the
canonical way. Recall that the category of étale spaces over $X$ is a topos equivalent to the category of sheaves on $X$. For
an arbitrary ultraconvergence space $X$, as we show in Proposition~\ref{prop:inf-pretopos}, the category
$\Et(X)$ is an infinitary pretopos. 
In other words, the category $\Et(X)$ enjoys all the exactness
properties of a topos but, in general, it may lack a generating set. 

We will prove this in the indexed description of étale spaces
as continuous maps $X \to \Set$ provided by Theorem
\ref{thm:etale-spaces-as-continuous-maps}, instead of the fibrational one. Indeed, all the infinitary-pretopos structure of $\Set^X$ extends to $\UltSp (X,
\Set)$ along the forgetful functor $\UltSp (X , \Set )\to \Set^X$. 
The following proposition, which is similar to \cite[Prop.~5.4.6]{lurie} for ultracategories and which can also be found in \cite[Prop.~5.7]{saadia2025}, makes this precise.
\begin{proposition}
\label{prop:inf-pretopos}
Let $X$ be an ultraconvergence space. The category $\UltSp (X, \Set)$ is an
infinitary pretopos, and the forgetful functor $\UltSp (X , \Set )\to \Set^X$
is a conservative infinitary-pretopos morphism which reflects finite limits, small coproducts, and epimorphisms.
\begin{proof}
We only prove here that $\UltSp (X, \Set)$ admits finite limits, which are both
preserved and reflected by the forgetful functor $U \colon \UltSp (X , \Set
)\to \Set^X$; the rest of the proof proceeds analogously.

Let $D \colon \mc C \to \UltSp(X,\Set)$ be a finite diagram in
$\UltSp(X,\Set)$. 
The limit of the composite diagram $UD \colon \mc C \to \Set^X$ is computed pointwise, i.e., it is $l \colon X \to \Set$ given by $l(x) = \lim_{c \in \mc C} D(c)(x)$ for each $x \in X$.
In order to endow $l$ with a continuity
structure, let $r \colon x \ult \lim_{i\to\mu} x_i$ be an ultra-arrow in $X$.
We need to define an ultra-arrow $l(r) \colon l(x) \ult \lim_{i\to \mu} l(x_i)$
in $\Set$, meaning a function $l(x) \to \prod_{i\to\mu} l(x_i)$, which -- since
finite limits commute with (limits and) filtered colimits in $\Set$ -- means
explicitly
\[ \lim_{\mc C} D(-) (x) \to \prod_{i\to\mu} \lim_{\mc C}D(-)(x_i) \iso
\lim_{\mc C} \left(\prod_{i\to\mu} D(-)(x_i)\right)\]
where $\prod_{i\to\mu} D(-)(x_i)$ is the ultraproduct of functors $\mc C \to
\Set$. Note then that such a function is determined by the universal property
of limits considering, for each $c \in \mc C$, the function $D(c)(x) \to
\prod_{i\to\mu} D(c)(x_i)$ given by the ultra-arrow $D(c)(r) \colon D(c)(x)
\ult \lim_{i\to\mu} D(c)(x_i)$. In other words, $l(r) \colon l(x) \ult
\lim_{i\to\mu} l(x_i)$ is the unique function $l(x) \to \prod_{i\to\mu} l(x_i)$
such that the diagram
\[\begin{tikzcd}
{l(x)} &[20pt] {D(c) (x)} \\
{\prod_{i\to\mu}l(x_i)} & {\prod_{i\to\mu}D(c)(x_i)}
\arrow["{(p_c)_x}", from=1-1, to=1-2]
\arrow[dashed, from=1-1, to=2-1]
\arrow["{D(c)(r)}", from=1-2, to=2-2]
\arrow["{\prod_{i\to\mu}(p_{c})_{x_i}}"', from=2-1, to=2-2]
\end{tikzcd}\]
commutes for each object $c$ in $\mc C$, where $p_c \colon l \implies D(c)$
denotes the projection natural transformation. This makes $l$ a continuous map
$X \to \Set$, and by construction it is the only continuity structure on $l$
such that each projection $p_c$ determines a continuous map $l \implies
D(c)$.

This definition makes $l$ a continuous map $X \to \Set$, and moreover the
projections $p_c \colon l \implies D(c)$ in $\Set^X$ can be identified with
morphisms of continuous maps $l \implies D(c)$, making $l$ (the vertex of) a
cone on $D$. By the universal property of limits, it is then immediate to see
that the universal natural transformation $\phi \colon l' \implies l$ in
$\Set^X$, for a cone $l' \colon X \to \Set$ on $D$, determines a morphism of
continuous maps $\phi \colon l' \implies l$, which is thus the universal
morphism in $\UltSp(X,\Set)$ making $l$ the limit of $D$. By construction, it
is also evident that $U \colon \UltSp(X,\Set) \to \Set^X$ preserves and
reflects finite limits.
\end{proof}
\end{proposition}

We conclude this section by characterizing subobjects in $\Et(B)$.
\begin{lemma}\label{lem:open-subobjets-of-etale-space}
Let $B$ be an ultraconvergence space and let $\pi\colon E \to B$ be an étale space. The subobjects of $\pi$ in
$\Et(B)$ are precisely the restrictions of $\pi$ to open subspaces of $E$.
\begin{proof}
    Subobjects of $\pi$ are, by definition, isomorphism classes of monomorphisms $m \colon \pi' \rightarrowtail \pi$, or, equivalently, subspaces $U$ of $E$ such that $\pi|_{U} \colon U \to B$ is étale. Let $U$ be a subspace of $E$. 
    Since $\pi$ is étale, we have in particular, for any $e \in U$ and ultra-arrow $r \colon \pi(e) \ult \lim_{i \to \mu} x_i$ in $B$, a unique lift $\bar{r} \colon e \ult \lim_{i \to \mu} e_i$ in $E$. Therefore, the restricted map $\pi|_{U}$ is étale if, and only if, for every $e \in U$ and $r \colon \pi(e) \ult \lim_{i \to \mu} x_i$, the lifted ultrafamily $(e_i)_{i \to \mu}$ is eventually in $U$. This condition is equivalent to $U$ being open: it clearly must hold if $U$ is open, and for the converse, if $s \colon e \ult \lim_{i \to \mu} e_i$ is any ultra-arrow in $E$ with $e \in U$, then applying the condition to $r = \pi(s)$ shows that the $e_i$ are $\mu$-eventually in $U$.
\end{proof}
\end{lemma}

\section{Étale spaces over the points of a topos}\label{sec:main}
In this section, we first construct a canonical ultraconvergence structure on any class of points $X$ of a topos $\mc{E}$. We then show in Construction~\ref{con:etale-of-object} how to associate with any object of $\mc{E}$ an étale space over this canonical ultraconvergence space $X$. This association extends to an infinitary-pretopos morphism, denoted by $\sem{-}$, from $\mc{E}$ to $\Et(X)$. The remainder of the section contains the proof of our main theorem, Theorem~\ref{thm:main-theorem}, which states that this morphism is an equivalence provided that $X$ is a separating set of points. We will give a detailed outline of the proof once the required definitions are in place.

\paragraph{The canonical ultraconvergence structure.}Let $\mu$ be an ultrafilter on some set $I$. As recalled in Section \ref{introduction}, by {\L}o{\'s}'s theorem, the ultraproduct functor $\prod_{i\to\mu} \colon \cSet^I \to \cSet$ is coherent, and hence a pretopos morphism. As a consequence, for a
pretopos $\mc T$, the class $\Mod {(\mc T)}$ of \emph{models} of $\mc T$, i.e.\ pretopos morphisms $\mc T \to \cSet$, is equipped with a canonical ultraconvergence structure, which we now describe.

For a $\mu$-family $(x_i)_{i\to\mu}$ of pretopos morphisms $\mc T \to \cSet$, the functor $\prod_{i\to\mu} x_i \colon \mc T \to \Set$ defined by applying the ultraproduct functor pointwise is again a pretopos morphism by {\L}o{\'s}'s theorem. Therefore, defining ultra-arrows $x \ult \lim_{i\to\mu} x_i$ to be natural transformations $x \implies \prod_{i\to\mu} x_i$ endows $\Mod (\mc T)$ with an ultraconvergence structure. This ultraconvergence structure restricts to any class $X$ of models of $\mc T$. In particular, taking $\mc T$ to be a topos $\mc E$ and $X$ to be a class of points of $\Ec$, we obtain the \emph{canonical ultraconvergence structure} on $X$.

\paragraph{Étale spaces from evaluations.}Let $\Ec$ be a topos and $X$ a class of points of $\Ec$, endowed with the canonical ultraconvergence structure: ultra-arrows $x \ult \lim_{i \to \mu} x_i$ are natural transformations $x \To \prod_{i \to \mu} x_i$. 

The evaluation functor $\eval{-} \colon \Ec \to \Set^X$ described in Section~\ref{introduction} lifts to a functor $\eval{-} \colon \Ec \to \UltSp(X,\Set)$: for any object $\phi$ of $\Ec$, its associated function $\eval{\phi} \colon X \to \Set$ carries a canonical continuity structure, sending any ultra-arrow $r \colon x \ult \lim_{i \to \mu} x_i$ to its component $r_\phi$. In the proof of our main theorem, we will mainly work with the \emph{étale space} corresponding to this continuous map $\eval{\phi} \colon X \to \Set$ through the equivalence of Theorem~\ref{thm:etale-spaces-as-continuous-maps}, which we denote $\pi_\phi \colon \sem{\phi} \to X$. For convenience, we now describe this étale space more directly.

\begin{constr}\label{con:etale-of-object}
    Let $\phi$ be an object of $\Ec$. We describe its \emph{associated étale space} \[ \pi_\phi \colon \eval{\phi} \to X \ . \]
The points of the ultraconvergence space $\eval{\phi}$ are pairs $(x, v)$ where $x \in X$ and $v \in x(\phi)$, and the function $\pi_\phi$ is the projection on the first coordinate. 
Given a point $(x, v) \in \eval{\phi}$ and an ultrafamily $(x_i, v_i)_{i\to\mu}$ in $\eval{\phi}$, we define an ultra-arrow $r \colon (x,v) \ult \lim_{i\to\mu} (x_i, v_i)$ to be an ultra-arrow $r \colon x\ult \lim_{i\to\mu} x_i$ in $X$ 
which satisfies $r_\phi (v) = (v_i)_{i\to\mu}$. 
We extend the function $\pi_\phi$ to ultra-arrows by setting $\pi_\phi(r)$ to be $r$ itself, which evidently makes $\pi_\phi$ continuous. 

Note that the map $\pi_\phi$ is étale: indeed, for any ultra-arrow $r \colon \pi_\phi(x, v)\ult \lim_{i\to\mu} x_i$ in $X$, denote by $(v_i)_{i \to \mu}$ the ultra-family $r_\phi(v)$. Then the unique lift of $r$ is the ultra-arrow $r \colon (x,v) \ult \lim_{i\to\mu} (x_i,v_i)$ in $\eval{\phi}$. 
\end{constr}
\begin{remark}
Let $\mathbb T$ be a geometric theory classified by the topos $\Ec$, so that we can think of $X$ as a class of models of $\mathbb T$ with the ultraconvergence structure described in Example \ref{ex:ultraconvergence-space-modT}. Recall that objects of $\Ec$ can be seen as sheaves over the syntactic category of $\mathbb T$, so that we can intuitively think of $\phi$ as a `geometric construction' over the theory $\mathbb T$. Then the fiber of $\pi_\phi \colon \eval{\phi}\to X$ at a model $x \in X$ is precisely the \emph{extension} of the formula $\phi$ in $x$. 
\end{remark}
Observe that, indeed, the continuous map associated by Theorem~\ref{thm:etale-spaces-as-continuous-maps} to the étale space $\pi_\phi$ of Construction~\ref{con:etale-of-object} is the evaluation function $\eval{\phi}\colon X \to \Set$, equipped with its canonical continuity structure. Since each point $x \in X$ is an infinitary-pretopos morphism, so is $\eval{-} \colon \Ec \to \Set^X$. Since the forgetful functor $\UltSp(X,\Set) \to \Set^X$ reflects the infinitary-pretopos structure (Proposition~\ref{prop:inf-pretopos}), it follows that $\sem{-} \colon \Ec \to \UltSp(X,\Set)$ is also an infinitary-pretopos morphism. Composing with the equivalence $\Et(X) \simeq \UltSp(X, \Set)$ of Theorem~\ref{thm:etale-spaces-as-continuous-maps}, we thus obtain an infinitary-pretopos morphism $\Ec \to \Et(X)$, which we also denote by $\eval{-}$. Explicitly, for an arrow $f \colon \phi \to \psi$ in $\Ec$, the associated morphism of étale spaces $\pi_\phi \to \pi_\psi$ is the continuous map sending a point $(x, v) \in \eval{\phi}$ to the point $(x, x(f)(v)) \in \eval{\psi}$, and acting as the identity on ultra-arrows.

A necessary condition for $\eval{-}$ to be an equivalence of categories is conservativity. As the forgetful functor $\UltSp(X,\Set) \to \Set^X$ is conservative, this is equivalent to the evaluation functor $\eval{-} \colon \Ec \to \Set^X$ being conservative, i.e.\ to the class $X$ being separating for the topos $\Ec$. In particular, this means that $\Ec$ must have enough points, so that we can assume without loss of generality that $X$ is a separating \emph{set} of points, as opposed to a proper class.

We are now ready to prove the main theorem of this paper:

\mainresult*

\paragraph{Proof outline.} Our proof of Theorem~\ref{thm:main-theorem} follows the same structure as Makkai's proof for the pretopos setting, see~\cite[Lem.~4.2]{makkai}. We will prove that the functor $\eval{-} \colon \Ec \to \Et(X)$ has the following two properties:
\begin{enumerate}
    \item $\eval{-}$ is \emph{full on subobjects}, i.e.\ for each $\phi$ in $\Ec$ and each subobject $p \colon Y \to X$ of $\pi_\phi \colon \sem{\phi} \to X$ in $\Et(X)$, there exists a subobject $\psi$ of $\phi$ in $\Ec$ such that $\eval{\psi} = p$ in $\Sub_{\Et(X)}(\eval{\phi})$; 
    \item $\eval{-}$ is \emph{covering}, i.e.\ for each étale space $\pi \colon Y \to X$, there exist an object $\phi$ in $\Ec$ and an epimorphism $\alpha \colon \pi_\phi \onto \pi$ in $\Et(X)$. 
\end{enumerate}
Given these two properties, Theorem~\ref{thm:main-theorem} can then be deduced from \cite[Lem.~1.4.9]{makkai-reyes}. For the reader's convenience, we give an explicit proof of this final step. Recall first that a conservative cartesian functor is faithful, and if it is full on subobjects then it is also full (see, e.g., \cite[Lem.~D3.5.6]{elephant}). 
        Therefore, to conclude, it remains to show that $\eval{-}$ is essentially surjective. To this end, let $\pi \colon Y \to X$ be an étale space. 
        As $\eval{-}$ is covering, pick a surjective continuous map $\alpha \colon \eval{\phi}\onto Y$ such that $\pi \circ \alpha = \pi_\phi$ for some $\phi $ in $ \Ec$. 
        Consider the kernel pair of $\alpha$ in $\Et(X)$, which we can identify with the subspace $R \coloneqq \set{ (a,b) \in \eval{\phi}^2 | \alpha(a) = \alpha(b)}$ of $\eval{\phi} \times_X \eval{\phi}$ together with the restriction of the étale map $\pi_\phi \times \pi_\phi \colon \eval{\phi} \times_X \eval{\phi} \to X$ to $R$. 
        As $\eval{-}$ preserves finite products and is full on subobjects, there exists a subobject $\rho \rightarrowtail \phi \times \phi$ in $\Ec$ such that $R = \eval{\rho}$. 
        Note then that $\eval{-}$ reflects equivalence relations, and hence $\rho$ is an equivalence relation in $\Ec$. 
        Thus, $\rho \rightarrowtail \phi \times \phi$ is the kernel pair of its coequalizer $\phi \onto \phi/\rho$, which $\eval{-}$ maps to $\alpha$; in particular, $Y = \eval{\phi/\rho}$.

Below, we will complete the proof by establishing that $\eval{-}$ is indeed full on subobjects (Proposition~\ref{prop:step-1}) and covering (Proposition~\ref{prop:step-2}).\medskip

\begin{tcolorbox}[emphbox]
\textbf{Assumption.} Throughout the remainder of this section, we fix a topos $\Ec$, a separating set $X$ of points of $\Ec$ endowed with the canonical ultraconvergence structure, and a small full subcategory $\mc{S}$ of $\Ec$ which is dense and closed under finite limits.
\end{tcolorbox}

\subsection{Fullness on subobjects}
The aim of this subsection is to prove Proposition~\ref{prop:step-1}, stating that $\sem{-}$ is full on subobjects. The proof strategy is to first prove it in the special case of subobjects of the terminal object (Lemma~\ref{lem:every-open-is-eval}), from which we then deduce the general fact by slicing.

The étale space associated to the terminal object $1$ of $\Ec$ is the identity map $1_X \colon X \to X$, which is indeed terminal in $\Et(X)$.
Restricting the functor $\eval{-} \colon \Ec \to \Et(X)$ to subterminals thus yields a function 
\[ \eval{-} \colon \Sub_{\Ec}(1) \to \Sub_{\Et(X)}(1_X),\]
which is a frame homomorphism since $\eval{-}$ is an infinitary-pretopos morphism,
and is injective since $\eval{-}$ is conservative.
By applying Lemma~\ref{lem:open-subobjets-of-etale-space} to the étale space $\id_X \colon X \to X$, 
the frame $\Sub_{\Et(X)}(1_X)$ can be identified with $\Open(X)$.
The main technical result of this subsection,
Lemma~\ref{lem:every-open-is-eval} below, says
that $\eval{-} \colon \Sub_{\Ec}(1) \to \Open(X)$ is also surjective,
and therefore an isomorphism $\Sub_{\Ec}(1) \iso \Open(X)$. As we will see in Proposition~\ref{prop:step-1},
this will suffice to ensure that $\eval{-} \colon \Ec \to \Et(X)$ is full on
subobjects by considering slice toposes.

In order to explain the model-theoretic intuition behind Lemma~\ref{lem:every-open-is-eval}, recall from Example~\ref{ex:ultraconvergence-space-modT} that we may think of $X$ as a set of models for a geometric theory $\mathbb{T}$. A subset $U$ of $X$ is then \emph{open} if, and only if, for any model $x$ in $U$ and any ultra-family of models $(x_i)_{i \to \mu}$ in $X$, if $x$ admits a homomorphism to $\prod_{i \to \mu} x_i$, then $x_i$ is $\mu$-eventually in $U$. It is also clear from this point of view that $\sem{\phi}$ is open for any geometric sentence $\phi$. The surjectivity of $\sem{-}$ means that, conversely, any such open class $U$ of models can be defined by a geometric sentence in the theory $\mathbb{T}$. 

\looseness=-1 
In the proof of Lemma~\ref{lem:every-open-is-eval}, we use the following. For any $x \in X$, if $\phi$ is subterminal in $\mc{E}$, then $x(\phi)$ is subterminal in $\Set$, and we have $x \in \sem{\phi}$ if, and only if, $x(\phi)$ is non-empty. We thus have, for any $x \in X$, a completely prime filter 
$F_x \coloneqq \set{ \phi \in \Sub_{\mc{E}}(1)  | x \in \sem{\phi} }$ in $\Sub_{\Ec}(1)$.

\begin{lemma}\label{lem:every-open-is-eval}
    Let $U$ be an open subspace of $X$. There exists $\phi_U \in \Sub_{\Ec}(1)$ such that $\sem{\phi_U} = U$.
\end{lemma}
\begin{proof}
    Consider $\phi_U \coloneqq \bigvee \set{ \phi \in \Sub_{\Ec}(1) \colon \sem{\phi} \subseteq U}$, a subterminal object of $\Ec$. Since the map $\sem{-} \colon \Sub_{\Ec}(1) \to \Open(X)$ is a frame homomorphism, the equality $\sem{\phi_U} = U$ is equivalent to
\begin{equation*}
\bigcup \set{ \sem{\phi} \subseteq X |\phi \in \Sub_{\mc{E}}(1)
\colon \sem{\phi} \subseteq U} = U.
\end{equation*}

    The left-to-right inclusion is clear. For the right-to-left inclusion, we will show the contrapositive. Let $x \in X$ and suppose that, for every $\phi \in \Sub_{\mc{E}}(1)$, if $\sem{\phi} \subseteq U$, then $x \not\in \sem{\phi}$. Equivalently, for every $\phi \in F_x$, the set $\sem{\phi} \setminus U$ is non-empty.
We will show $x \not\in U$.

Write $\mc{I} \coloneqq \int \restr{x}{\mc{S}}$
for the category of elements of the restriction of the point $x \colon \Ec \to
\Set$ to $\mc{S}$, which is small as so is $\mc S$. By Lemma
\ref{lem:cofiltered-category-of-elements}, as $\mc S$ is finitely complete,
$\mc{I}$ is cofiltered. By Lemma \ref{lem:initial-functor-cofiltered}, 
pick a cofiltered poset $\mc{P}$
and an initial functor $d \colon \mc{P} \to \mc{I}$. For each $p \in \mc P$,
let us write $d(p) \coloneqq (\psi_p, v_p)$, where $\psi_p$ is an object of
$\mc{S}$ and $v_p$ is an element of the set $x(\psi_p)$.

Let $p \in \mc{P}$. Write $\exists \psi_p$ for the image of the unique
arrow $!_{\psi_p} \colon \psi_p \to 1$. Note that $\exists \psi_p$ is in $F_x$: 
    $v_p$ is in the domain of the function $x(!_{\psi_p})$, so its image $x(\exists \psi_p)$ is non-empty.  
    Therefore, by hypothesis we can choose a point
$x_{p} \in \sem{\exists \psi_p} \setminus U$. Since $x_p$ sends the epimorphism
    $\psi_p \onto \exists \psi_p$ in $\mc{E}$ to a surjective function, we can also pick
    an element $w_p \in x_p(\psi_p)$.

    Choose an ultrafilter $\mu$ on $\mc P$ which contains every principal down-set of $\mc{P}$; such an ultrafilter exists because $\mc{P}$ is cofiltered. In order to construct an ultra-arrow $r \colon x \ult \lim_{p \to \mu} x_p$, by Lemma~\ref{lem:extension-from-a-dense-subcategory}, it suffices to define a natural transformation $r \colon x|_{\mc{S}} \To \prod_{p \to \mu} x_p|_{\mc{S}}$, as we will do now. 

Let $\psi$ be an arbitrary object of $\mc{S}$ and let $u \in x(\psi)$.
By the initiality of $d \colon \mc{P} \to \mc{I}$, 
we can pick $p \in \mc{P}$ and an arrow
$f \colon (\psi_p, v_p) \to (\psi,u)$ in $\mc I$. 
For each
$q \leq p$, we then also have an arrow $\restr f q \coloneqq f \circ d(q \leq
    p) \colon (\psi_q, v_q) \to (\psi, u)$ in $\mc{I}$, which means that $\restr{f}{q} \colon \psi_q \to \psi$ is an arrow in $\mc{E}$ such that $x(\restr{f}{q})(v_q) = u$. We define
$r_\psi(u)$ to be the $\mu$-family determined by the sequence $(x_q(\restr f
q)(w_q))_{q \leq p}$; this defines an element of the ultraproduct
$\prod_{p \to \mu} x_p(\psi)$ since ${\downarrow} p \in \mu$.
To see that this definition does not depend on the choice of $f$, let $f'
\colon (\psi_{p'}, v_{p'}) \to (\psi, u)$ be an arrow in $\mc{I}$, for some $p'
\in \mc{P}$. By initiality of $d$, pick $p_0 \in \mc{P}$ with $p_0 \leq p$ and
$p_0 \leq p'$ such that $\restr f {p_0}  = \restr {f'}{p_0}$. Then, for any $q
\leq p_0$, we have $\restr f q = \restr{f'}q$, so that the sequences
$(x_q(\restr f q)(w_q))_{q \leq p}$ and $(x_q(\restr{f'}q)(w_q))_{q \leq p'}$
coincide on the set ${\downarrow} p_0$, which is in $\mu$.

Finally, we prove that $r$ is natural in $\psi$. Let $g \colon \psi \to
\psi'$ be an arrow in $\mc{S}$. We will show that the following diagram
commutes:
\begin{equation}\label{eq:naturality-r}
\begin{tikzcd}
x(\psi) \ar[r, "r_\psi"] \ar[d, "x(g)"'] & \prod_{p \to \mu} x_p(\psi)
\ar[d, "\prod_{p \to \mu} x_p(g)"] \\
x(\psi') \ar[r, "r_{\psi'}"'] & \prod_{p \to \mu} x_p(\psi').
\end{tikzcd}
\end{equation}
Let $u \in x(\psi)$. Write $u' := x(g)(u)$, so that $g$ determines an arrow
$(\psi, u) \to (\psi', u')$ in $\mc{I}$. According to the definition above, we have, for any arrows $f \colon (\psi_p,v_p) \to (\psi, u)$ and $f' \colon (\psi_{p'}, v_{p'}) \to (\psi',u')$ in $\mc{I}$,
\[ r_\psi(u) = \left( x_q(\restr f q)(w_q) \right)_{q \to \mu} \quad \text{ and } \quad r_{\psi'}(u') = \left(x_q(\restr{f'}q)(w_q)\right)_{q \to \mu} \ .\]
Taking $f' = g \circ f$, we obtain
\[ r_{\psi'}(u') = \left( x_q(g \circ \restr f q)(w_q) \right)_{q\to\mu} =
\Big( x_q(g)(r_\psi(u)) \Big)_{q\to\mu}, \]
which is precisely the square \eqref{eq:naturality-r} at $u \in x(\psi)$.

To conclude, note that we have constructed an ultra-arrow $r \colon x \ult \lim_{p \to \mu} x_p$. Since $x_p \not\in U$ for every $p \in \mathcal{P}$, it follows that $x \not\in U$, since $U$ is open.
\end{proof}

In order to prove fullness on subobjects, we now apply Lemma~\ref{lem:every-open-is-eval} in slice toposes of $\mc{E}$. 
\begin{proposition}\label{prop:step-1}
The functor $\eval{-}\colon \Oc(X) \to \Et(X)$ is full on subobjects.
\end{proposition}
\begin{proof}
Let $\phi $ be an object of $\Ec$ and let $U \subseteq \eval{\phi}$ be an open subspace. 
Recall (e.g.\ by~\cite[Prop.~IV.5.12]{SGA4}) that points of the
slice topos $\Ec/\phi$ may be identified with pairs $(e, u)$ where $e$ is a point of $\Ec$ 
and $u \in e(\phi)$. Note that the set of such pairs whose
first component lies in $X$ is separating for $\Ec/\phi$, and coincides with the set
$\eval{\phi}$ underlying the étale space associated with $\phi$. 
Write $\eval{-}_\phi \colon \Ec/\phi \to \Et(\eval{\phi})$ for the evaluation functor
associated with the topos $\Ec/\phi$ and the set of points $\eval{\phi}$. 
Applying Lemma~\ref{lem:every-open-is-eval} in this case, there exists a subterminal $\psi$ in $\Ec/\phi$ such that $U = \eval{\psi}_{\phi}$. The subterminal $\psi$ in $\Ec/\phi$ may be identified with a subobject $\psi$ of $\phi$, and we then have that the subspace $\sem{\psi}$ of $\sem{\phi}$ coincides with $\eval{\psi}_{\phi} = U$.
\end{proof}

\subsection{Covering}
We will now prove that $\sem{-}$ is covering. 
The crucial technical point is Lemma~\ref{lem:local-injective-pullback} below.
In light of Proposition~\ref{prop:parallel-lifts-and-local-injectivity}, this lemma will show
that any étale space over $X$ can be `locally trivialized' by taking a pullback with an object
of the form $\sem{\phi}$. Here, instantiating Construction~\ref{constr:pullbacks} 
when $\pi \colon Y \to X$ is an étale space over $X$ and $\phi$ is an object of $\Ec$, the pullback
\[ \begin{tikzcd}
  |[label={[label distance=-2mm]-45:\lrcorner}]|{Y \times_X \sem{\phi}} & {\eval{\phi}} \\
Y & X
\arrow["\widetilde{\pi_\phi}",from=1-1, to=2-1]
\arrow["\widetilde{\pi}"', from=1-1, to=1-2]
\arrow["\pi", from=2-1, to=2-2]
\arrow["{\pi_\phi}"', from=1-2, to=2-2]
\end{tikzcd} \]
consists of pairs $(y, a)$ with $a \in \pi(y)(\phi)$.

Lemma~\ref{lem:local-injective-pullback} below is analogous to \cite[Lem.~4.4]{makkai}, where a pair $(\phi,a)$ as in the statement of the lemma is called a \emph{support} of $y$.
\begin{lemma}\label{lem:local-injective-pullback}
    Let $\pi \colon Y \to X$ be an étale space, $x \in X$, and $y \in \pi^{-1}(x)$. There exist an object $\phi$ of $\Ec$ and $a \in x(\phi)$ such that the pullback $\widetilde{\pi} \colon Y \times_X \sem{\phi} \to \sem{\phi}$ is locally injective at $(y, a)$.
\end{lemma}

\begin{proof}
    As in the proof of Lemma~\ref{lem:every-open-is-eval}, we write $\mc{I} \coloneqq \int x|_{\mc{S}}$, we pick an initial functor $d \colon \mc{P} \to \mc{I}$ with $\mc{P}$ a cofiltered poset, and we choose an ultrafilter $\mu$ on $\mc{P}$ that contains every principal down-set of $\mc{P}$. For each $p \in \mc{P}$, write $d(p) \coloneqq (\psi_p, v_p)$, where $\psi_p \in \mc{S}$ and $v_p \in x(\psi_p)$.

    Now, towards a contradiction, suppose that no $\phi$ and $a$ with the claimed property exist. Then, in particular, for every $p \in \mc{P}$, the pullback $Y \times_X \sem{\psi_p} \to \sem{\psi_p}$ is not locally injective at $(y, v_p)$. By Proposition~\ref{prop:parallel-lifts-and-local-injectivity}(\ref{it:par-lift}), pick an ultrafamily $(x^p_i)_{i \to \nu_p}$ in $X$ and natural transformations $r^p \colon x \To \prod_{i \to \nu_p} x^p_i$ and $s^p \colon x \To \prod_{i \to \nu_p} x^p_i$ such that $r^p(v_p) = s^p(v_p)$, but $r^p(y) \neq s^p(y)$.

Consider the composite ultra-arrows
\[ r \coloneqq (r^p)_{p \to \mu} \cdot \id_x[!] \text{ and } s \coloneqq (s^p)_{p \to \mu} \cdot \id_x[!], \]
where $\id_x[!] \colon x \ult \lim_{p \to \mu} x$ is the ultra-arrow in $X$ obtained by reindexing the identity ultra-arrow $\id_x \colon x \ult \lim_{\ast \to 1} x$ along the unique $\UF$-arrow $! \colon (\mc{P}, \mu) \to (\{\ast\}, 1)$.
We will show that $r(y) \neq s(y)$, but also that $r = s$, which is a contradiction.

By composition, $r(y) = (r^p_i(y))_{p\to\mu;i\to\nu_p}$ and $s(y) = (s^p_i(y))_{p\to\mu;i\to\nu_p}$. We know that $r^p_i(y) \neq s^p_i(y)$ holds $\nu_p$-eventually for every fixed $p$, by hypothesis. Hence, $r(y) \neq s(y)$.

It remains to show that $r=s$. By Lemma~\ref{lem:extension-from-a-dense-subcategory}, it is enough to show that $r_\phi = s_\phi$ for every $\phi \in \mc S$. Let $\phi \in \mc S$ and let $u \in x(\phi)$. Since $d$ is initial, there is an arrow $h \colon (\phi_p,v_p) \to (\phi,u)$ in $\mc{I}$ for some $p \in \mc P$, which means that $h \colon \phi_p \to \phi$ satisfies $x(h)(v_p) = u$. Reusing the notation $\restr{h}{q} = h \circ d(q\leq p)$ from Lemma~\ref{lem:every-open-is-eval}, we obtain
\begin{align*}
	r(u) &= r(x(h)(v_p)) = (r^q_i(x(h)(v_p)))_{q\to\mu;i\to\nu_q} \tag*{by definition}\\
	&= (r^q_i(x(\restr{h}{q})(v_q)))_{q\to\mu;i\to\nu_q} \tag*{since $q \leq p$ holds $\mu$-eventually}\\
	&= (x^q_i(\restr{h}{q})(r^q_i(v_q)))_{q\to\mu;i\to\nu_q} \tag*{by naturality of $r^q$}\\
	&= (x^q_i(\restr{h}{q})(s^q_i(v_q)))_{q\to\mu;i\to\nu_q} \tag*{since $r^q(v_q) = s^q(v_q)$}\\
	&= s(u) \tag*{by the same argument reversed}
\end{align*}
We conclude that $r = s$, but this contradicts $r(y) \neq s(y)$.
\end{proof}
\begin{proposition}\label{prop:step-2}
    The functor $\sem{-} \colon \Ec \to \Et(X)$ is covering.
\end{proposition}
\begin{proof}
    \looseness=-1 
    Let $\pi \colon Y \to X$ be an étale space, and let $y \in Y$ be arbitrary. By Lemma~\ref{lem:local-injective-pullback}, pick an object $\phi$ of $\mc{E}$ and $a \in \pi(y)(\phi)$ such that $\widetilde{\pi}$ is locally injective at $(y, a)$. By Proposition~\ref{prop:parallel-lifts-and-local-injectivity}(\ref{it:loc-inj}),  there exists an open subset $V \subseteq Y \times_X \sem{\phi}$ such that $\restr {\widetilde{\pi}} {V} \colon V \to \sem{\phi}$ is injective, and $(y, a) \in V$. 
The subset $\widetilde{\pi}[V]$ of $\sem{\phi}$ is open by Lemma \ref{lem:etale-open}: so, using that $\sem{-}$ is full on subobjects (Proposition~\ref{prop:step-1}), pick a subobject $\psi_y \rightarrowtail
\phi$ in $\mc E$ such that $\sem{\psi_y} = \widetilde{\pi}[V]$ as subspaces of $\eval{\phi}$.
By Lemma \ref{lem:bijective-étale-is-isomorphism}, $\restr{\widetilde{\pi}}V \colon V \to
\sem{\psi_y}$ is an isomorphism in $\UltSp$. Define $\alpha_y \colon \eval{\psi_y}
\to Y$ to be the composite
\[ \begin{tikzcd}
    {\eval{\psi_y}} &[20pt] V & {Y \times_X \eval{\phi}} & Y.
\arrow["{(\restr{\widetilde{\pi}}V)^{-1}}", from=1-1, to=1-2]
\arrow[hook, from=1-2, to=1-3]
\arrow["\widetilde{\pi_\phi}", from=1-3, to=1-4]
\end{tikzcd} \]
Note that $\pi \circ \alpha_y = \restr{\pi_\phi}{\eval{\psi_y}} \circ \restr{\widetilde{\pi}}{V} \circ
(\restr {\widetilde{\pi}} V)^{-1} = \restr{\pi_\phi}{\eval{\psi_y}} = \pi_{\psi_y}$, so that $\alpha$ is a morphism from 
the étale space $\pi_{\psi_y}$ to the étale space $\pi$.
Moreover, $\alpha_y(\widetilde{\pi}(y,a)) = \widetilde{\pi_\phi}(y,a) = y$, so $y$ lies in the image of $\alpha$.

Thus, we may choose, for every $y \in Y$, an object $\psi_y$ of $\mc{E}$ and a morphism of étale spaces $\alpha_y \colon \pi_{\psi_y} \to \pi$ whose image contains $y$. The coproduct $\psi \coloneqq \sum_{y \in Y} \psi_y$ is an object of $\mc{E}$, and we have a surjective morphism $\alpha \colon \sem{\psi} \cong \sum_{y \in Y} \sem{\psi_y} \onto \pi$, obtained as the coproduct of the morphisms $\alpha_y$.
\end{proof}

\section{Conclusion}\label{sec:conclusions}

In this paper, we have shown how the canonical ultraconvergence structure on a separating set of points of a topos allows us to recover the topos as a category of étale spaces. Although the theorem is a direct extension of Lurie's \cite{lurie} reconstruction theorem for coherent toposes, our proof is structurally closer to Makkai's \cite{makkai} proof of his reconstruction theorem for pretoposes via ultracategories, which it generalizes and simplifies. As highlighted by the proof of the key \Cref{prop:parallel-lifts-and-local-injectivity}, our proof requires the fixed set of points to be, indeed, a \emph{set} rather than a proper class. However, as also noted in \cite[Thm.~8.3]{saadia2025}, a simple argument allows one to lift this to a reconstruction theorem for toposes with enough points, presenting such a topos as the category of étale spaces over its (large) ultraconvergence space of points.

\begin{corollary}\label{cor:reconstruction}
Let $\Ec$ be a topos with enough points. Then, the functor $\sem{-} \colon
\Ec \to \Et(\pt (\Ec))$ is an equivalence of categories.
\begin{proof}
Let $\mc J$ be the filtered posetal category of separating sets of points of $\Ec$, ordered by
inclusion. Then, $\pt (\Ec)$ is the pseudocolimit of the diagram defined by
$\mc J$ in $\CAT$, and hence the ultraconvergence space $\pt (\Ec)$ is the
pseudocolimit of the diagram defined by $\mc J$ in $\UltSp$ with the canonical
ultraconvergence structures. By Theorem \ref{thm:main-theorem}, for each $X \in
\mc J$ we have an equivalence $\sem{-} \colon \Ec \to \UltSp(X,\Set)$, from which:
\[ \UltSp ( \pt ( \Ec) , \Set ) \simeq\UltSp ( \colim_{X \in \mc
J}\pt ( \Ec) , \Set ) \simeq \lim_{X\in \opp{\mc J} } \UltSp ( X, \Set) \simeq
\lim_{X\in \opp{\mc J} } \Ec \simeq \Ec , \]
making $\sem{-} \colon \Ec \to \UltSp(\pt(\Ec), \Set)$ an equivalence.
\end{proof}
\end{corollary}

\begin{example}\label{ex:alexandroff-presheaf-type}
      Building on Examples~\ref{ex:ultraconvergence-space-set} and \ref{ex:ultraconvergence-space-modT}, note that the object classifier $\Set[\mathbb O]$ is equivalent to $\Et(\Set) \simeq \UltSp(\Set,\Set)$. It is well known that $\Set[\mathbb O]$ is also equivalent to $[\cat{Fin},\Set]$ where $\cat{Fin}$ is the category of finite sets \cite[Sec.~D3.2]{elephant}. This can be understood in terms of ultraconvergence spaces as follows: the finite sets form a separating class of models of $\Set[\mathbb O]$, hence $\Set[\mathbb O]$ is equivalent to $\UltSp(\cat{Fin},\Set)$. But since $\cat{Fin}$ is Alexandroff (Construction~\ref{con:alexandroff}), the continuous maps $\cat{Fin}\to\Set$ are just the functors $\cat{Fin}\to\Set$. The same story can be told about any theory of presheaf type \cite{BekTheoriesPresheafType2004}, which are dual to the (Cauchy-complete) Alexandroff ultraconvergence spaces.
\end{example}

As in the coherent case, we can also rephrase our result in terms of $2$-fully-faithfulness.
\begin{corollary}\label{cor:adjunction}
      The 2-category $\mathbf{Topos}_{wep}$ of toposes with enough points embeds 2-fully-faithfully into the 2-category $\mathbf{UltSp}$ of ultraconvergence spaces.
      \begin{proof}
            Let $\mathbf{UltSp}_0$ be the full 2-subcategory of $\mathbf{UltSp}$ spanned by ultraconvergence spaces $X$ such that $\UltSp(X,\Set)$ is a topos, and note that $\Set$ lies in $\mathbf{UltSp}_0$ as  $\UltSp(\Set,\Set)$ is equivalent to the topos $\Set[\mathbb O]$. Seen as a dualizing object, $\Set$ yields a 2-adjunction
            \[ \begin{tikzcd}
                  {\mathbf{Topos}} && {\UltSp_0}
                  \arrow[""{name=0, anchor=center, inner sep=0}, "{\pt{(-)}}"', curve={height=18pt}, from=1-1, to=1-3]
                  \arrow[""{name=1, anchor=center, inner sep=0}, "{\UltSp(-, \Set)}"', curve={height=18pt}, from=1-3, to=1-1]
                  \arrow["\dashv"{anchor=center, rotate=-90}, draw=none, from=1, to=0]
            \end{tikzcd} \]
            whose counit, at a topos $\Ec$, is given by (a geometric morphism whose inverse image is) the functor $\eval{-} \colon \Ec \to \UltSp(\pt (\Ec), \Set)$. By Theorem \ref{thm:main-theorem}, this is an equivalence for each topos with enough points, meaning that the 2-adjunction induces a 2-fully-faithful embedding $\mathbf{Topos}_{wep} \hookrightarrow \UltSp_0$. In particular, $\mathbf{Topos}_{wep}$ embeds 2-fully-faithfully into $\mathbf{UltSp}$.
      \end{proof}
\end{corollary}

\begin{remark}
\looseness=-1
From a logical point of view, the previous corollaries can be interpreted as a \emph{strong conceptual completeness} theorem, in the sense of Makkai \cite{makkai,makkai-strong-cc-fol}, for geometric logic. Despite its name, note that this does not imply a \emph{conceptual completeness} theorem for geometric logic: that is, it is not true that a geometric morphism $\mc E \to \mc F$ inducing an equivalence of categories $\pt (\Ec) \to \pt(\mc F)$ is an equivalence itself, not even for toposes with enough points. 

In other words, conceptual completeness is a \emph{conservativity} property over $\CAT$, in the 2-dimensional sense of reflecting \emph{equivalences} rather than isomorphisms, for an appropriate 2-functor taking categories of models: in the setting of coherent logic, this was originally proved in \cite[Chap.~7]{makkai-reyes}, identifying coherent theories with pretoposes. Instead, \emph{strong} conceptual completeness is a \emph{reconstruction} result of syntax from semantics, and it typically requires considering additional structure on the categories of models: in the setting of coherent logic, this is the main result of \cite{makkai,lurie}, allowed by ultrastructure. A general definition of the latter kind, for fragments of geometric logic, is proposed in \cite[\S 6.3]{diliberti2025logicconcepts2categorytopoi}, where it is simply referred to as \emph{conceptual completeness}.
\end{remark}

\begin{remark}
	Via Construction~\ref{con:alexandroff}, we can identify small categories with Alexandroff ultraconvergence spaces whose underlying category is small. By \Cref{rem:sp-and-alex} we know that the topos $\UltSp ( \mc C, \Set)$ induced by a small category $\mc C$ is simply the presheaf topos $\Et(\mc{C}) \coloneqq [\mc{C},\cSet]$. However, for small categories $\mc{C}$ and $\mc{D}$, there are in general more geometric morphisms $\Et(\mc{C}) \to \Et(\mc{D})$ than continuous maps (i.e.\ functors) $\mc{C} \to \mc{D}$. Indeed, even assuming $\mc{D}$ to be Cauchy-complete, the only geometric morphisms arising from functors  $\mc{C} \to \mc{D}$ are the \emph{essential} ones \cite[Ex.\ A4.1.4]{elephant}. In general, geometric morphisms $\Et(\mc{C}) \to \Et(\mc{D})$ correspond to functors $\mc{C} \to \pt(\Et(\mc{D}))$, where the category $\pt(\Et(\mc{D}))$ is equivalent to the \emph{Ind-completion} $\Ind(\mc{D})$ of $\mc D$ by \cite[Prop.\ 0.1]{BekTheoriesPresheafType2004}. 
    
    In other words, we can think of the Alexandroff space $\Ind(\mc{D})$ as the \emph{soberification} of a small category $\mc D$, in analogy with the topological case, since for any topos $\mc F$ we have:
	\[ \mathbf{UltSp}(\mc D ,\pt(\mc F)) \simeq \mathbf{Topos}(\Et(\mc D),\mc F) \simeq \mathbf{UltSp}(\Ind(\mc{D}) ,\pt(\mc F)) \text{.} \]
    \looseness=-1
    More generally, for a topos $\mc E$, the ultraconvergence space $\pt (\mc E)$ can be seen as the soberification of any subclass $X$ of points endowed with the canonical ultraconvergence structure. This raises the question of constructing the soberification of $X$ directly in terms of the space, without going through the topos $\mc E$. 
    We leave this for future work.
\end{remark}

\printbibliography 
 
\end{document}